\documentclass[article,authoryear,jmlmc]{beg_32}             

\usepackage[hang]{footmisc}
\fancypagestyle{plain}{%
  \fancyhf{}
  \fancyhead[R]{\small {\it \jname}, x(x):\thepage--\pageref{LastPage} (\myyear\today)}
  \fancyfoot[R]{\small\bf\thepage }
  \fancyfoot[L]{\fottitle}
  }
\usepackage{graphicx}%
\usepackage{multirow}%
\usepackage{amsmath,amssymb,amsfonts}%
\usepackage{amsthm}%
\usepackage{mathrsfs}%
\usepackage{xcolor}%
\usepackage{textcomp}%
\usepackage{manyfoot}%
\usepackage{booktabs}%
\usepackage{listings}%
\usepackage[utf8]{inputenc}
\usepackage[T1]{fontenc}
\usepackage{hyperref}
\usepackage{url}
\usepackage{bm}
\usepackage{nicefrac}
\usepackage{microtype}
\usepackage{verbatim}
\usepackage{svg}
\usepackage{array}
\usepackage{placeins}

\newcommand{\Dom}{\mathrm{Dom}\,}

\DeclareMathOperator{\Ran}{Ran}

\newcommand{\Tr}{\operatorname{Tr}}

\newcommand{\Mall}{\mathcal{D}} 
\newcommand{\Frechet}{\nabla} 
\newcommand{\CMspace}{\mathcal{H}_t} 

\newtheorem{proposition}[theorem]{Proposition}%
\newtheorem{definition}{Definition}%

\renewcommand{\dmyy}{07}
\renewcommand{\myyear}{2026}
\renewcommand{\today}{}
\begin{document}
\volume{Volume x, Issue x, \myyear\today}
\title{Score-Based Diffusion Models in Infinite Dimensions: A Malliavin Calculus Perspective}
\titlehead{Score-Based Diffusion Models in Infinite Dimensions}
\authorhead{E. Mirafzali, F. Proske, D. Venturi, \& R. Marinescu}
\corrauthor[1]{Ehsan Mirafzali}
\author[2]{Frank Proske}
\author[3]{Daniele Venturi}
\author[1]{Razvan Marinescu}
\corremail{smirafza@ucsc.edu}
\corraddress{Department of Computer Science, University of California Santa Cruz, Santa Cruz, California, USA}
\address[1]{Department of Computer Science, University of California Santa Cruz, Santa Cruz, California, USA}
\address[2]{Department of Mathematics, University of Oslo, Oslo, Norway}
\address[3]{Department of Applied Mathematics, University of California Santa Cruz, Santa Cruz, California, USA}
\dataO{01/07/2026}
\dataF{01/07/2026}
\abstract{We study score-based diffusion modelling in infinite-dimensional separable Hilbert spaces through Malliavin calculus, extending the analysis of generative models beyond the finite-dimensional setting. The forward diffusion process is formulated as a linear stochastic partial differential equation (SPDE) driven by space--time coloured noise with a trace-class covariance operator, ensuring well-posedness in arbitrary spatial dimensions. Building on Malliavin calculus and an infinite-dimensional extension of the Bismut--Elworthy--Li formula, we derive a closed-form expression for the logarithmic derivative of the transition measure along Cameron--Martin directions, which serves as the natural infinite-dimensional analogue of the score function. Our operator-theoretic approach preserves the intrinsic geometry of Hilbert spaces and accommodates general trace-class operators, thereby incorporating spatially correlated noise without assuming semigroup invertibility. We validate the derived score formula numerically for several classes of linear SPDEs in both one and two spatial dimensions using spectral methods.}
\maketitle

\section{Introduction}\label{sec1}
Recent advances in diffusion generative models \citep{pmlr-v37-sohl-dickstein15, ho2020denoising, Song2021} have substantially advanced data synthesis. 
These models rely on the so-called score function \citep{Hyvarinen2005, HYVARINEN20072499, Song2019, ni2025divergencekernelmethodscoresrandom}, the gradient of the log-density, to generate an iterative denoising process that reverses a forward diffusion process, achieving state-of-the-art performance in many practical applications such as image and audio generation, and design of new molecular structures. 
Classical score-based diffusion models are defined in finite, yet often very high-dimensional, spaces. Extending these models to infinite dimensions, such as those arising in functional data analysis, presents significant mathematical challenges.
For instance, infinite-dimensional spaces lack a Lebesgue measure, making traditional density gradients ill-defined. Instead, one must work with logarithmic derivatives (or Fomin derivatives) of measures along appropriate directions, which are well-defined only in the Cameron--Martin space associated with the underlying Gaussian measure. Moreover, for diffusion processes governed by stochastic partial differential equations (SPDEs), the noise must often be regularised to guarantee well-posedness of the initial-boundary value problem. In the case of abstract SPDEs of the form 
\begin{equation}
\label{eq:spde}
du(t) = A u(t)\,dt + Q^{1/2} dW_t, \quad u(0)=u_0\in H,
\end{equation}
where $H$ is a separable Hilbert space, $A$ is a densely defined unbounded linear operator (e.g., the Laplacian with Dirichlet or Neumann boundary conditions), and $W_t$ is a cylindrical Wiener process, this regularisation is achieved by taking $Q^{1/2}: U \to H$ to be Hilbert-Schmidt\footnote{More precisely, \( W_t \) is a cylindrical Wiener process on a separable Hilbert space \(U\) (hence not \(U\)-valued), and \( Q^{1/2}: U \to H \) is a Hilbert-Schmidt operator. The induced covariance operator $Q := Q^{1/2}(Q^{1/2})^* \in \mathcal{L}_1(H)$ is then trace-class on $H$.} \citep{daprato2014stochastic}, so that $Q^{1/2} dW_t$ defines spatially correlated (coloured) noise.

The main objective of this paper is to study score-based diffusion modelling in infinite dimensions using Malliavin calculus~\citep{malliavin:78:stochastic}. Our analysis extends the recent work \citep{mirafzali2025malliavincalculusscorebaseddiffusion}; \citep{mirafzali2025malliavincalculusapproachscore} from finite-dimensional diffusion generative models -- where it has proved effective in deriving exact closed-form expressions for the score function -- to the infinite-dimensional setting of linear SPDEs with additive Gaussian noise. 
Our work is closely related to that of~\citep{greco2025malliavingammacalculusapproachscore}, who studied a similar problem in infinite-dimensional score-based diffusion generative models using Malliavin calculus. In their formulation, the score function is identified (via time reversal) as a Malliavin derivative and corresponds to a conditional expectation, and the approach uses Gaussian Ornstein--Uhlenbeck noising defined via Dirichlet forms and is built around $\Gamma$-calculus. Likewise, \citep{pidstrigach2025conditioning} employ Malliavin calculus for conditioned diffusions to derive a Tweedie-like formula for the score function, although their focus remains on finite-dimensional conditioning mechanisms rather than on score computation in infinite dimensions.
On the other hand, our approach relies on an operator-theoretic derivation of the logarithmic derivative (directional score) for linear SPDEs of the form \eqref{eq:spde} with a general Hilbert-Schmidt operator \(Q^{1/2}: U \to H\), without resorting to any spatial discretisation. By leveraging techniques for differentiating Hilbert-valued processes~\citep{Nualart_Nualart_2018} together with infinite-dimensional extensions of the Bismut--Elworthy--Li formula~\citep{bakhtin2007malliavincalculusinfinitedimensionalsystems,elworthy1994}, we obtain closed-form expressions for the logarithmic derivative in the infinite-dimensional setting.
Other recent contributions formulate diffusion models directly on infinite-dimensional function spaces—either by perturbing functions with a Gaussian process specified by a covariance kernel and deriving a resolution–independent discretised algorithm \citep{lim2025scorebaseddiffusionmodelsfunction}, by casting the dynamics as stochastic evolution equations in Hilbert spaces and implementing them via spatial discretisation \citep{lim2023scorebased}, or by developing an infinite-dimensional formulation with dimension–independent guarantees \citep{pidstrigach2024}. None of these approaches employ Malliavin calculus.

This paper is organised as follows. 
In Section~\ref{sec:notation} we establish notation and define the relevant spaces and topologies. In Section~\ref{sec:methodology} we present our methodology, including the operator-theoretic derivation of the logarithmic derivative for linear SPDEs with space--time coloured noise, the Malliavin covariance operator, and the infinite-dimensional Bismut--Elworthy--Li formula. In Section~\ref{sec:experiments} we verify the score formula numerically in both one and two spatial dimensions. In Section~\ref{sec:summary}, we discuss our main findings and outline possible avenues for future work. 
For the reader's convenience, we provide~\ref{app:background} where we review the relevant mathematical background, including infinite-dimensional diffusion processes generated by linear SPDEs and Malliavin calculus.

\section{Notation and Preliminaries}
\label{sec:notation}

We begin by establishing notation and defining the relevant Hilbert spaces, operator spaces, and topologies used throughout this paper. This section is essential for ensuring mathematical consistency.

\subsection{Hilbert Spaces and Inner Products}

Throughout this paper, we work with the following separable Hilbert spaces:
\begin{itemize}
    \item $H$: the \emph{state space}, typically $L^2(V)$ for some spatial domain $V \subset \mathbb{R}^d$, equipped with inner product $\langle \cdot, \cdot \rangle_H$ and norm $\|u\|_H = \sqrt{\langle u, u \rangle_H}$;
    \item $U$: an auxiliary separable Hilbert space on which the cylindrical Wiener process is defined, equipped with inner product $\langle \cdot, \cdot \rangle_U$ and norm $\|\cdot\|_U$;
    \item $H_W := L^2([0,t]; U)$: the space of square-integrable $U$-valued functions on $[0,t]$, equipped with the inner product
    \begin{equation}
    \langle \eta, \zeta \rangle_{H_W} = \int_0^t \langle \eta(r), \zeta(r) \rangle_U \, dr.
    \label{eq:HW-inner-product}
    \end{equation}
\end{itemize}
All closures of subsets of these Hilbert spaces are taken with respect to the corresponding Hilbert space norm unless otherwise specified.

\subsection{Operator Spaces}

We denote by:
\begin{itemize}
    \item $\mathcal{L}(X, Y)$: the space of bounded linear operators from Hilbert space $X$ to Hilbert space $Y$, with operator norm $\|\cdot\|_{\mathcal{L}(X,Y)}$;
    \item $\mathcal{L}_{\mathrm{HS}}(X, Y)$: the space of \emph{Hilbert-Schmidt operators} from $X$ to $Y$, with norm
    \begin{equation}
    \|T\|_{\mathrm{HS}}^2 = \sum_{k=1}^\infty \|T e_k\|_Y^2,
    \end{equation}
    where $\{e_k\}_{k=1}^\infty$ is any orthonormal basis of $X$ (the sum is independent of the choice of basis);
    \item $\mathcal{L}_1(H)$: the space of \emph{trace-class operators} on $H$, consisting of operators $T \in \mathcal{L}(H)$ for which $\Tr(|T|) < \infty$.
\end{itemize}
For a bounded operator $T \in \mathcal{L}(X, Y)$, we denote its adjoint by $T^* \in \mathcal{L}(Y, X)$.

\subsection{Derivatives}

To avoid confusion, we use distinct notation for different types of derivatives:
\begin{itemize}
    \item $\Frechet \phi$: the \emph{Fr\'echet derivative} of a functional $\phi: H \to \mathbb{R}$, identified via the Riesz representation theorem with an element of $H$;
    \item $\Mall_r F$: the \emph{Malliavin derivative} of a random variable $F$ at time $r \in [0,t]$, which takes values in $U$ for scalar-valued $F$, or in $\mathcal{L}_{\mathrm{HS}}(U, H)$ for $H$-valued $F$;
    \item $D_{u_0} u(t)$: the Fr\'echet derivative of the SPDE solution with respect to the initial condition, which for the linear SPDE~\eqref{eq:spde} equals the semigroup $S(t)$ (see Section~\ref{sec:methodology}, Eq.~\eqref{FVP}).
\end{itemize}

\subsection{The Cameron--Martin Space}
\label{sec:CM-space}

A central concept in infinite-dimensional analysis is the Cameron--Martin space, which provides the natural domain for logarithmic derivatives. For a Gaussian measure $\mu$ on $H$ with covariance operator $\gamma$, the Cameron--Martin space is defined as $\CMspace := \Ran(\gamma^{1/2})$, equipped with the inner product
\begin{equation}
\langle h_1, h_2 \rangle_{\CMspace} := \langle \gamma^{-1/2} h_1, \gamma^{-1/2} h_2 \rangle_H,
\label{eq:CM-inner-product}
\end{equation}
where $\gamma^{-1/2}$ denotes the inverse of $\gamma^{1/2}$ on its range. When $\gamma$ is trace-class (as is the case for our Malliavin covariance operator), the Cameron--Martin space $\CMspace$ is strictly smaller than $H$ (indeed, it has $\mu$-measure zero), but it is precisely the space of directions along which the Gaussian measure admits a logarithmic derivative; see \citep{bogachev2015gaussian, daprato2014stochastic}.

While $\CMspace = \Ran(\gamma_t^{1/2})$ may appear ``extremely small'' compared to $H$, it is in fact the \emph{maximal} space on which the score is well-defined. For typical choices of $A$ and $Q$ (e.g., $A = \Delta$ with Dirichlet boundary conditions and $Q$ diagonal in the eigenbasis of $-A$), $\CMspace$ consists of functions with higher Sobolev regularity than generic elements of $H = L^2$. In practical implementations using finite-dimensional discretisations or projections onto the first $n$ eigenmodes, the projected Cameron--Martin space becomes all of $\mathbb{R}^n$, and our operator formulae correspond exactly to the continuum limit.

\section{Methodology}
\label{sec:methodology}

In this section, we derive a closed-form expression for 
the score function associated with the solution of linear SPDEs of the form \eqref{eq:spde}, driven by space--time (trace class) coloured noise.
Our analysis is based on the following key well-posedness assumptions:

\begin{enumerate}
\item \( H \) is a separable Hilbert space, typically \( L^2(V) \) for some spatial domain \( V \subset \mathbb{R}^d \) (e.g., \( V = (0,1) \) in one dimension), equipped with the inner product \( \langle u, v \rangle_H \) and norm $\left\| u \right\|_H = \sqrt{\left\langle u, u \right\rangle_H}$; \vspace{0.2cm}\label{Afirst}

\item \( A: D(A) \subset H \to H \) is a densely defined, unbounded linear operator (e.g., the Laplacian \( \Delta \) with Dirichlet or Neumann boundary conditions) generating a strongly continuous semigroup \( S(t) = e^{tA} \) on \( H \), satisfying $\left\| S(t) u \right\|_H \leq M e^{\omega t} \left\| u \right\|_H$ for some $M \geq 1$, $\omega \in \mathbb{R}$; \vspace{0.2cm}

\item \( W_t \) is a cylindrical Wiener process on an auxiliary separable Hilbert space \( U \); we do not assume \(U = H\). The process is defined on the filtered probability space \( (\Omega, \mathcal{F}, \{\mathcal{F}_t\}_{t \geq 0}, \mathbb{P}) \) and formally expressed via an orthonormal basis \( \{ f_k \}_{k=1}^\infty \) of \( U \) as
\begin{equation}
W_t = \sum_{k=1}^\infty f_k B_t^k,
\label{F1}
\end{equation}
where \( \{ B_t^k \}_{k=1}^\infty \) are independent standard Brownian motions\footnote{As is well known, the series \eqref{F1} may not 
converge in \( U \) but defines a cylindrical process in a larger 
space.}. \vspace{0.2cm}

\item \( Q^{1/2}: U \to H \) is a Hilbert-Schmidt operator with norm 
\[
\left\| Q^{1/2} \right\|_{\mathrm{HS}}^2 = \sum_{k=1}^\infty \left\| Q^{1/2} f_k \right\|_H^2 < \infty,
\]
where \( \{ f_k \}_{k=1}^\infty \) is an orthonormal basis of \( U \). \vspace{0.2cm}

\item \( u_0 \in H \) is a deterministic initial condition.\vspace{0.2cm}\label{Alast}
\end{enumerate}
The strong form, for intuition, is
\[
\partial_t u\left(t, x\right) = A u\left(t, x\right) + \eta\left(t, x\right), \quad u\left(0, x\right) = u_0\left(x\right), \quad x \in V
\]
with \( \eta(t, x) = Q^{1/2} \xi(t, x) \) and appropriate boundary conditions (e.g., \( u(t, x) = 0 \) on \( \partial V \)). The mild solution is
\[
u(t) = S(t) u_0 + \int_0^t S(t - s) Q^{1/2} \, dW_s,
\]
where the stochastic integral is the Hilbert space It\^o integral in $L^2(\Omega;H)$. For well-definedness, \( Q^{1/2}: U \to H \) is Hilbert-Schmidt, and we require
\begin{equation}\label{eq:HS-condition}
\int_0^t \left\|S(s) Q^{1/2}\right\|_{\mathrm{HS}}^2 \, ds < \infty,
\end{equation}
where
\[
\left\|S(s) Q^{1/2}\right\|_{\mathrm{HS}}^2 = \sum_{j=1}^\infty \left\|S(s) Q^{1/2} f_j\right\|_H^2
\]
for an orthonormal basis \(\{f_j\}_{j=1}^\infty\) of \(U\).

Our objective now is to derive the Malliavin covariance operator of the SPDE solution \( u(t) \). This operator measures the stochastic sensitivity of \( u(t) \) with respect to perturbations in the space--time coloured noise across the entirety of \( H \), providing insights into the regularity and density properties of the law of \( u(t) \) (e.g., via the absolute continuity with respect to a Gaussian measure in \( H \)). With the noise term \( Q^{1/2} dW_t \), this sensitivity is with respect to the coloured noise, reflecting the spatial covariance structure induced by \( Q^{1/2} \), unlike the white noise case where spatial correlations are absent. This, in turn, will allow us to compute a closed-form expression for the logarithmic derivative of the transition measure of infinite-dimensional diffusion processes satisfying \eqref{eq:spde}.

\subsection{Malliavin Covariance Operator}

To compute the Malliavin covariance operator associated with \( u(t) \), we express it in terms of the first variation process, which for state-independent diffusion (i.e., \( Q^{1/2} \)), coincides with the semigroup \( S(t) \). 
This mirrors the finite-dimensional stochastic differential equation (SDE) case \( dx_t = b(x_t) dt + \sigma dW_t \), where the first variation process for state-independent diffusion coincides with the solution of a deterministic linear differential equation driven by the Jacobian of the drift, i.e., \( Y_t = {\partial x_t}/{\partial x_0} \) satisfies \( dY_t = b'(x_t) Y_t dt \), $Y_0=I$. 
For the linear SPDE \eqref{eq:spde} we have the formal solution
\begin{equation}
u(t) = S(t) u_0 + \int_0^t S(t-s) Q^{1/2} dW_s.
\label{sol}
\end{equation}
The Fr\'echet derivative of $u(t)$ with respect to the initial condition $u_0$ is \citep{Daniele2021}
\begin{equation}
D_{u_0}u(t) = S(t).
\end{equation}
Hence, the first variation process for the SPDE \eqref{eq:spde} satisfies 
\begin{equation}
dY_t = A Y_t dt,\quad Y_0=I_H\quad \Rightarrow \quad Y_t = S(t).
\label{FVP}
\end{equation} 
Note that for each $v\in H$
\begin{equation}
\left\|Y_t v\right\|_H = \left\|S(t) v\right\|_H \leq \left\|S(t)\right\|_{\mathcal{L}(H)} \left\|v\right\|_H
\end{equation}
which is finite for all $t \geq 0$.

Hereafter we leverage this linearity to define the Malliavin covariance operator of $u(t)$ directly on \( H \). The following theorem provides the precise formulation of this result. 

\begin{theorem}
\label{thm:malliavin-covariance-spde}
Consider the linear SPDE \eqref{eq:spde} with the assumptions at the beginning of Section \ref{sec:methodology}. Define the first variation process $Y_t = S(t)$, i.e., the semigroup generated by $A$. The Malliavin covariance operator $\gamma_t: H \to H$ of the solution $u(t)$ at time $t > 0$ is given by
\begin{equation}
\gamma_t = \int_0^t S(s) Q^{1/2}\left(Q^{1/2}\right)^* S(s)^* \, ds,
\label{MCO}
\end{equation}
where $S(s)^*$ is the adjoint of $S(s)$ in $\mathcal{L}(H)$.
The operator $\gamma_t$ is positive, self-adjoint, and trace-class on $H$.

If $A$ generates a strongly continuous group $\{S(t)\}_{t \in \mathbb{R}}$, then $S(t)$ is invertible for all $t\in\mathbb{R}$ and the Malliavin covariance operator can be alternatively expressed as
\[
\gamma_t = Y_t C_t Y_t^*,\qquad  C_t= \int_0^t Y_r^{-1} Q^{1/2}\left(Q^{1/2}\right)^* \left(Y_r^{-1}\right)^* \, dr,
\]
where $Y_r^{-1} = S(-r)$ and the integral $C_t$ is a trace-class operator on $H$.
\end{theorem}

\begin{proof}
We begin by establishing existence of the mild solution to our SPDE \eqref{eq:spde}. Since the equation is driven by space--time coloured noise, the noise operator $Q^{1/2}: U \to H$ is a Hilbert-Schmidt operator. The mild solution is given by
\begin{equation}
u(t) = S(t) u_0 + \int_0^t S(t - s) Q^{1/2} dW_s,
\label{Sint}
\end{equation}
where the stochastic integral is taken in $H$ with respect to the cylindrical Wiener process $W_t$. The well-posedness of this integral requires the integrability condition
\begin{equation}
\int_0^t \left\| S(s) Q^{1/2} \right\|_{\mathrm{HS}}^2 \, ds < \infty,
\end{equation}
which is satisfied under our assumptions. This condition ensures that $u(t) \in L^2(\Omega; H)$ and that the stochastic integral is well-defined.

To derive the Malliavin covariance operator, we first compute the Malliavin derivative $\Mall_r$ (at time $r$) of $u(t)$. This derivative measures the sensitivity of the solution $u(t)$ with respect to perturbations in the noise $W_r$. Since $u_0$ is deterministic, we have
\[
\Mall_r \left[S(t) u_0\right] = 0.
\]
For the stochastic integral term in \eqref{Sint}, the Malliavin derivative 
at time $r \in [0, t]$ is given by the integrand evaluated at the perturbation time (see~\ref{app:a3})
\begin{equation}
\Mall_r u(t) = S(t - r) Q^{1/2} \mathbf{1}_{[0,t]}(r),
\label{e1}
\end{equation}
where $\mathbf{1}_{[0,t]}(r)$ is the indicator function ensuring causality. This result follows from the standard theory of Malliavin calculus for stochastic integrals: if $\Phi(s)$ is a deterministic integrand, then $\Mall_r \left(\int_0^t \Phi(s) dW_s\right) = \Phi(r) \mathbf{1}_{[0,t]}(r)$. In our case we have $\Phi(s) = S(t - s) Q^{1/2}$, yielding the expression \eqref{e1}. Note that for each $r \in [0,t]$, $\Mall_r u(t) \in \mathcal{L}_{\mathrm{HS}}(U, H)$.
Next, define the Malliavin covariance operator 
\begin{equation}
\gamma_t =\int_0^t \Mall_r u(t) \left(\Mall_r u(t)\right)^* \, dr,
\label{MC}
\end{equation}
where $(\Mall_r u(t))^*: H \to U$ denotes the adjoint and $\Mall_r u(t)(\Mall_r u(t))^* \in \mathcal{L}(H)$. Substituting \eqref{e1} into \eqref{MC} we obtain 
\[
\gamma_t = \int_0^t S(t - r) Q^{1/2} \left(S(t - r) Q^{1/2}\right)^* \, dr.
\]
Since $Q^{1/2}: U \to H$ is a Hilbert-Schmidt operator with adjoint $(Q^{1/2})^*: H \to U$, we have
\[
\left(S(t - r) Q^{1/2}\right)^* = \left(Q^{1/2}\right)^* S(t - r)^*.
\]
This gives us
\[
\Mall_r u(t) (\Mall_r u(t))^* = S(t - r) Q^{1/2}\left(Q^{1/2}\right)^{*} S(t - r)^*,
\]
which in turn allows us to write \eqref{MC} as
\[
\gamma_t = \int_0^t S(t - r) Q^{1/2}\left(Q^{1/2}\right)^* S(t - r)^* \, dr.
\]
To obtain the final form of the Malliavin covariance operator, 
we perform the change of variables $s = t - r$. This gives
\begin{align}
\gamma_t = \int_0^t S(s) Q^{1/2}\left(Q^{1/2}\right)^* S(s)^* \, ds.
 \label{MCfin}
\end{align}
The covariance operator $\gamma_t$ is positive, self-adjoint, and trace-class on $H$. The trace-class property follows from the fact that $Q^{1/2}$ is Hilbert-Schmidt, $S(s)$ and $S(s)^*$ are bounded operators, and the integral is over a finite interval $[0, t]$. The integrability condition \eqref{eq:HS-condition} ensures these properties are preserved.

For the alternative representation of $\gamma_t$ when $A$ generates a $C_0$-group, we observe that if $S(t)$ forms a strongly continuous group, then $S(t)$ is invertible for all $t \in \mathbb{R}$ with $S(t)^{-1} = S(-t)$. Setting $Y_t = S(t)$ and $Y_r^{-1} = S(-r)$, we can factor the Malliavin covariance operator as
\begin{equation}
\gamma_t = \int_0^t S(t - r) Q^{1/2}\left(Q^{1/2}\right)^* S(t - r)^* \, dr = Y_t C_tY_t^*,
\label{MCgroup}
\end{equation}
where 
\begin{equation}
 C_t= \int_0^t Y_r^{-1} Q^{1/2}\left(Q^{1/2}\right)^* \left(Y_r^{-1}\right)^* \, dr.
\end{equation}
The integral $C_t$ in \eqref{MCgroup} represents the accumulated noise covariance transformed by the inverse semigroup, while the outer factors $Y_t$ and $Y_t^*$ propagate this covariance to time $t$.

\end{proof}

To illustrate the connection between the infinite-dimensional setting discussed in Theorem \ref{thm:malliavin-covariance-spde} and its finite-dimensional counterpart, let $\{ h_i \}_{i=1}^m \subset H$ be a set of linearly independent vectors, and define 
\[
F = \left( \langle u(t), h_1 \rangle_H, \dots, \langle u(t), h_m \rangle_H \right) \in \mathbb{R}^m.
\]
The Malliavin covariance matrix of $F$ can be computed as
\[
\gamma_F = \left( \langle h_i, \gamma_t h_j \rangle_H \right)_{i,j=1}^m,
\]
which provides a finite-dimensional representation of the infinite-dimensional operator $\gamma_t$.
Note that the Malliavin covariance operator $\gamma_t$ encapsulates both the spatial correlation induced by $Q^{1/2}$ and the temporal evolution governed by the semigroup $S(t)$, providing a characterisation of the stochastic variability of the solution in the infinite-dimensional Hilbert space $H$.

\subsection{The Cameron--Martin Space and Logarithmic Derivatives}
\label{sec:CM-and-logderiv}

Before proceeding to the Bismut formula, we must address a fundamental issue: in infinite dimensions, there is no Lebesgue measure on $H$, and hence the notion of a ``density'' $p_{u(t)}(u)$ with respect to which one computes a gradient requires careful interpretation. The appropriate framework is that of logarithmic derivatives (also called Fomin derivatives) of measures.

Let $\mu_t$ denote the law of $u(t)$ on $H$. Since $u(t)$ is Gaussian with mean $m_t := S(t) u_0$ and covariance operator $\gamma_t$, the measure $\mu_t$ is a Gaussian measure $\mathcal{N}(m_t, \gamma_t)$ on $H$. The Cameron--Martin space associated with $\mu_t$ is
\begin{equation}
\CMspace := \Ran(\gamma_t^{1/2}),
\label{eq:CM-space-def}
\end{equation}
equipped with the Cameron--Martin inner product \eqref{eq:CM-inner-product} (with $\gamma = \gamma_t$). Note that $\gamma_t^{-1/2}: \CMspace \to H$ is \emph{unbounded} when viewed as an operator on $(H, \|\cdot\|_H)$, but when $\CMspace$ is equipped with the Cameron--Martin norm $\|h\|_{\CMspace} := \|\gamma_t^{-1/2} h\|_H$, the map $\gamma_t^{-1/2}: (\CMspace, \|\cdot\|_{\CMspace}) \to (H, \|\cdot\|_H)$ is an isometry.

\begin{definition}[Logarithmic derivative]
\label{def:log-derivative}
Let $\mu_t$ be the law of $u(t)$ on $H$. For $h \in \CMspace$, the logarithmic derivative of $\mu_t$ along $h$ is the function $\beta_h: H \to \mathbb{R}$ defined $\mu_t$-almost everywhere by the integration-by-parts formula
\begin{equation}
\int_H \langle \Frechet \phi(u), h \rangle_H \, d\mu_t(u) = -\int_H \phi(u) \beta_h(u) \, d\mu_t(u)
\label{eq:log-deriv-def}
\end{equation}
for all $\phi \in C_b^1(H)$ (bounded continuously Fr\'echet differentiable functions).
\end{definition}

The logarithmic derivative $\beta_h(u)$ serves as the infinite-dimensional analogue of the directional derivative $\langle \nabla \log p(u), h \rangle$ in finite dimensions. The key insight is that $\beta_h$ is well-defined only for directions $h \in \CMspace$; for $h \notin \CMspace$, the measure $\mu_t$ is not quasi-invariant under translation by $h$, and no logarithmic derivative exists. Importantly, this restriction is not a limitation but rather reflects the intrinsic geometry of Gaussian measures in infinite dimensions. Translations along directions outside $\CMspace$ move the measure to a mutually singular measure, so there is no meaningful notion of a ``score'' in those directions. This is the precise infinite-dimensional analogue of the fact that in finite dimensions, the score $\nabla \log p(x)$ can only be computed at points where $p(x) > 0$. For practitioners, this means when discretising to $n$ modes, the projected Cameron--Martin space $\Pi_n \CMspace$ fills out $\mathbb{R}^n$ for any finite $n$, and the restriction becomes invisible. Our infinite-dimensional formulae give the continuum limit that these discretised scores converge to.

\subsection{Bismut Formula}
\label{sec:Bismut}
In this section, we extend the Bismut formula originally 
developed for finite-dimensional stochastic differential equations~\citep{alma990005308880204808, elworthy1994, Elworthy_1982, Nualart_Nualart_2018} to the infinite-dimensional setting of SPDEs of the form~\eqref{eq:spde}. This generalisation is then used to express the logarithmic derivative of the transition measure $\mu_t$, i.e., the infinite-dimensional analogue of the score function.

For directions $h \in \Ran(\gamma_t)$ where we can construct an explicit covering vector field, the Bismut formula provides a stochastic representation for the logarithmic derivative. Define $G(u) := \mathbb{E}[\delta(v_h) \mid \sigma(u(t))]$, which is a $\sigma(u(t))$-measurable random variable. The Bismut formula states that
\begin{equation}
\beta_h(u(t)) = -G(u(t)) \quad \mu_t\text{-a.e.},
\label{scoreBismut}
\end{equation}
where $\delta(v_h)$ is the Skorokhod integral of the covering vector field $v_h$, defined as the adjoint of the Malliavin derivative operator $\Mall_r$ (see~\ref{app:a5}). See \citep{Nualart_Nualart_2018} for an exhaustive treatment. Unlike the It\^o integral, the Skorokhod integral extends to non-adapted processes, capturing the effects of stochastic perturbations in the infinite-dimensional noise structure. 
The direction $h \in \Ran(\gamma_t)$ corresponds to an admissible perturbation along which the explicit covering construction applies. By density of $\Ran(\gamma_t)$ in $\CMspace = \Ran(\gamma_t^{1/2})$ (in the Cameron--Martin norm), the logarithmic derivative extends uniquely to all $h \in \CMspace$ as an element of $L^2(\mu_t)$; see Theorem~\ref{thm:score}.

Our next step is to define the covering vector field $v_h(r)$ appearing in the Skorokhod integral $\delta(v_h)$ within the SPDE framework. This definition accounts for the spatially correlated noise introduced by $Q^{1/2}$, ensuring consistency with the SPDE and its mild solution in $H$.

\begin{definition}[Covering Vector Field]
\label{def:covering-field}
For each $h \in \Ran(\gamma_t)$, define $\tilde{h} := \gamma_t^{-1} h \in H$ (noting that $\gamma_t^{-1}: \Ran(\gamma_t) \to H$ is densely defined but unbounded). The covering vector field $v_h: [0,t] \to U$ is defined by
\begin{equation}
v_h(r) = \left(Q^{1/2}\right)^{*} S(t-r)^{*}\,\tilde{h},\qquad r\in[0,t].
\label{defC}
\end{equation}
\end{definition}

Since $\Ran(\gamma_t)$ is dense in $\CMspace = \Ran(\gamma_t^{1/2})$ with respect to the Cameron--Martin norm, the covering construction and the resulting score formula extend by continuity to all $h \in \CMspace$. For practical computations (e.g., finite-rank truncations where $\gamma_t^{(n)}$ has finite rank), one has $\Ran(\gamma_t^{(n)}) = \Ran((\gamma_t^{(n)})^{1/2})$, so this distinction disappears in approximations.

This definition ensures that $v_h$ is a $U$-valued process reflecting the noise's covariance structure, with $Q^{1/2}$ modulating spatial dependence and $\gamma_t^{-1} h$ adjusting $h$ per the solution's stochastic dependence. In the following theorem we prove that \eqref{defC} satisfies the covering property
\begin{equation}
\langle \Mall u(t), v_h \rangle_{H_W} = h,
\label{CP}
\end{equation}
where $H_W = L^2([0,t]; U)$ is the space of square-integrable $U$-valued functions with inner product \eqref{eq:HW-inner-product}. Here, $\langle \Mall u(t), v_h \rangle_{H_W}$ denotes the $H$-valued Bochner integral
\begin{equation}
\langle \Mall u(t), v_h \rangle_{H_W} := \int_0^t \Mall_r u(t) \, v_h(r) \, dr,
\label{pairing-def}
\end{equation}
where $\Mall_r u(t) \in \mathcal{L}_{\mathrm{HS}}(U, H)$ acts on $v_h(r) \in U$ to produce an element of $H$.

\begin{theorem}[Covering property]
\label{thm:covering-property-spde}
Assume \ref{Afirst}--\ref{Alast} and \eqref{eq:HS-condition}. 
Let $u(t)$ be the mild solution of \eqref{eq:spde}, let $\gamma_t$ be given by \eqref{MCO}, and let $H_W=L^2([0,t];U)$. 
For each $h\in \Ran(\gamma_t)$, define $v_h$ by \eqref{defC} with $\tilde{h} = \gamma_t^{-1} h$. 
Then the covering property holds:
\[
\langle \Mall u(t), v_h \rangle_{H_W} = h.
\]
\end{theorem}

\begin{proof}
Consider the space $H_W = L^2([0,t]; U)$, where $U$ is the Hilbert space representing the input noise. In the context of the Malliavin derivative $\Mall_r u(t)$, which in this linear additive setting is a deterministic function of $r$ with values in 
$\mathcal{L}_{\mathrm{HS}}(U, H)$, and the vector field $v_h \in H_W$, the pairing $\left\langle \Mall u(t), v_h \right\rangle_{H_W}$ is defined as in \eqref{pairing-def}.

By \eqref{e1}, $\Mall_r u(t)=S(t-r)Q^{1/2}\mathbf{1}_{[0,t]}(r)$.
Let $h \in \Ran(\gamma_t)$ and write $\tilde{h} = \gamma_t^{-1} h$. Substituting the covering vector field \eqref{defC} into the 
pairing \eqref{pairing-def} we obtain 
\begin{equation}
\langle \Mall u(t), v_h \rangle_{H_W} = 
\int_0^t S(t - r) Q^{1/2} \left( \left(Q^{1/2}\right)^{*} S(t - r)^* \tilde{h} \right) \, dr.
\label{t1}
\end{equation}
Here, $S(t - r) Q^{1/2}$ maps $U$ to $H$, and $\left(Q^{1/2}\right)^{*} S(t - r)^* \tilde{h}$ is a vector in $U$ ensuring 
the composition is well‑defined.
To evaluate \eqref{t1}, we perform the change of variables $s = t - r$. This yields
\begin{align}
\int_0^t S(t - r) Q^{1/2}(Q^{1/2})^* S(t - r)^* \tilde{h} \, dr 
= \int_0^t S(s) Q^{1/2}(Q^{1/2})^* S(s)^* \tilde{h} \, ds.
\end{align}
From Theorem~\ref{thm:malliavin-covariance-spde}, we know that the Malliavin covariance operator can be expressed as in \eqref{MCO}. Thus, the pairing simplifies to
\begin{align}
\langle \Mall u(t), v_h \rangle_{H_W} 
&= \int_0^t S(s) Q^{1/2}(Q^{1/2})^* S(s)^* \tilde{h} \, ds \nonumber\\
&= \left( \int_0^t S(s) Q^{1/2}(Q^{1/2})^* S(s)^* \, ds \right) \tilde{h} \nonumber\\
&= \gamma_t \tilde{h} = \gamma_t \gamma_t^{-1} h = h.
\end{align}

\end{proof}

\begin{theorem}[Solution structure and minimal norm]
\label{thm:uniqueness-spde}
Under the assumptions of Theorem~\ref{thm:covering-property-spde}, define the subspace $\mathcal V:=\{(\Mall u(t))^{*}z:z\in H\}\subset H_W$ and fix $h\in\Ran(\gamma_t)$. Then:
\begin{enumerate}
\item There is a unique $v\in\mathcal V$ with $\langle \Mall u(t),v\rangle_{H_W}=h$, namely $v_h$ as defined in \eqref{defC}.
\item Among all solutions $v\in H_W$ of the covering equation $\langle \Mall u(t), v \rangle_{H_W} = h$, the covering vector field $v_h$ is the unique solution of minimal $H_W$-norm.
\item Every solution in $H_W$ is of the form $v_h+w$ with $w\in\mathcal V^\perp$, and
\[
\left\|v_h+w\right\|_{H_W}^2=\left\|v_h\right\|_{H_W}^2+\left\|w\right\|_{H_W}^2.
\]
\end{enumerate}
\end{theorem}
\begin{proof}
(1) Let $v\in\mathcal V$ solve the covering equation. Then $v=(\Mall u(t))^{*}z$ for some $z\in H$, and
\[
\langle \Mall u(t),v\rangle_{H_W}
=\int_0^t \Mall_r u(t)\,(\Mall u(t))^{*}z\,dr
=\gamma_t z
= h .
\]
For $h \in \Ran(\gamma_t)$, there exists unique $\tilde{h} = \gamma_t^{-1} h \in H$ with $\gamma_t \tilde{h} = h$. Then $\gamma_t z = h$ implies $z - \tilde{h} \in \ker(\gamma_t)$. Since $\gamma_t = \Mall u(t)(\Mall u(t))^*$ is positive self-adjoint, $\ker\gamma_t = \ker(\Mall u(t))^*$. Thus
\[
v-(\Mall u(t))^{*}\tilde{h}
=(\Mall u(t))^{*}\!\bigl(z-\tilde{h}\bigr) = 0,
\]
as $(\Mall u(t))^*$ annihilates $\ker\gamma_t$. Therefore $v = v_h$, proving uniqueness in $\mathcal V$.

(2) Let $T:H_W\to H$ be $Tv:=\langle \Mall u(t),v\rangle_{H_W}$ (Bochner integral). Then $\ker T=\mathcal V^\perp$ because for any $z\in H$,
\[
\langle z, Tv\rangle_H=\int_0^t\!\!\langle \Mall_r u(t)^*z,\,v(r)\rangle_U\,dr
=\langle (\Mall u(t))^{*}z,\,v\rangle_{H_W}.
\]
Thus every solution $v$ can be written uniquely as $v=v_\parallel+v_\perp$ with $v_\parallel\in\mathcal V$, $v_\perp\in\mathcal V^\perp$ and $Tv_\perp=0$, $Tv_\parallel=h$. Among all such decompositions,
\[
\left\|v\right\|_{H_W}^2=\left\|v_\parallel\right\|_{H_W}^2+\left\|v_\perp\right\|_{H_W}^2
\]
is minimised iff $v_\perp=0$. By part (1), $v_\parallel = v_h$ is unique, so the minimal-norm solution is uniquely $v_h$.

(3) If $v$ is any solution, then $T(v-v_h)=0$, so $v-v_h\in\ker T=\mathcal V^\perp$. Thus $v = v_h + w$ for some $w\in\mathcal V^\perp$. The norm identity follows from orthogonality $\mathcal V \perp \mathcal V^\perp$.
\end{proof}

\subsection{Skorokhod Integral}

In this section we show that, for the linear SPDE \eqref{eq:spde} with state-independent colouring \(Q^{1/2}\) (Hilbert–Schmidt), the Skorokhod integral simplifies to the more tractable It\^o integral. To this end, let us first recall the Bismut formula \eqref{scoreBismut}, which for $h \in \Ran(\gamma_t)$ states that
\[
\beta_h(u(t)) = -\mathbb{E} \left[ \delta(v_h) \mid \sigma(u(t)) \right] \quad \mu_t\text{-a.e.},
\]
expressing the logarithmic derivative in terms of the conditional expectation of the Skorokhod integral $\delta(v_h)$. The Skorokhod integral $\delta: \Dom(\delta) \subset L^2(\Omega; H_W) \to L^2(\Omega)$ is defined as the adjoint of the Malliavin derivative $\Mall$ via the duality relation
\begin{equation}
\mathbb{E}[\langle \Mall F, v \rangle_{H_W}] = \mathbb{E}[F \delta(v)]
\label{eq:Skorokhod-duality}
\end{equation}
for all $F \in \mathbb{D}^{1,2}$ and $v \in \Dom(\delta)$; see~\ref{app:a5} for details. For deterministic or adapted integrands $v_h \in H_W$ satisfying $\mathbb{E}[\int_0^t \|v_h(r)\|_U^2 \, dr] < \infty$, the Skorokhod integral coincides with the It\^o integral:
\begin{equation}
\delta(v_h) = \int_0^t \langle v_h(r), dW_r \rangle_U.
\label{Skorokhod}
\end{equation}

Our goal now is to show that, due to the linearity of the drift term $A u(t)$ in the SPDE \eqref{eq:spde} and the state-independence of the noise $Q^{1/2} dW_t$, the covering vector field $v_h(r)$ is adapted to the filtration $\{\mathcal{F}_r\}_{r \geq 0}$. In fact, $v_h(r)$ is \emph{deterministic}, which implies adaptedness trivially. This simplifies the Skorokhod integral to an It\^o integral with respect to the cylindrical Wiener process $W_t$
\begin{equation}
\delta(v_h) = \int_0^t \langle v_h(r), dW_r \rangle_U.
\label{Ito}
\end{equation}
A substitution of \eqref{defC} into \eqref{Ito} yields
\begin{equation}
\delta(v_h) = \int_0^t \left<\left (Q^{1/2}\right)^* S(t - r)^* \tilde{h}, dW_r \right>_U,
\label{Ito1}
\end{equation}
where $\tilde{h} = \gamma_t^{-1} h \in H$ for $h \in \Ran(\gamma_t)$.

In the next theorem we show that the covering vector field $v_h(r)$ associated with the linear SPDE \eqref{eq:spde} is indeed deterministic and adapted to the filtration $\{\mathcal{F}_r\}_{r \geq 0}$. Consequently, the Skorokhod integral~\eqref{Skorokhod} coincides with the It\^o integral~\eqref{Ito1}.

\begin{theorem}
\label{thm:SkorokhodIto}
The covering vector field \eqref{defC} for the linear SPDE \eqref{eq:spde} is a deterministic, $U$-valued process that is adapted to the filtration $\{\mathcal{F}_r\}_{r \geq 0}$, thereby satisfying the necessary conditions for It\^o integration in the infinite-dimensional setting. Consequently, the Skorokhod integral~\eqref{Skorokhod} coincides with the It\^o integral
\begin{equation}
\delta(v_h) = \int_0^t \langle v_h(r), dW_r \rangle_U \quad \text{and} \quad \langle \Mall u(t), v_h \rangle_{H_W} = \int_0^t \Mall_r u(t) v_h(r) \, dr.
\end{equation}
Moreover, $v_h \in \Dom(\delta)$ with $\mathbb{E}[\|\delta(v_h)\|^2] = \|v_h\|_{H_W}^2 < \infty$.
\end{theorem}

\begin{proof}
Consider the covering vector field $v_h(r)$ defined in \eqref{defC}, where $\tilde{h} = \gamma_t^{-1} h \in H$ for $h \in \Ran(\gamma_t)$.
To evaluate the properties of $v_h(r)$ in the space $U$, we first notice that $\tilde{h} \in H$ is deterministic, and $S(t - r)^*: H \to H$ can be bounded as
\begin{align}
\left\|S(t - r)^* \tilde{h}\right\|_H &\leq \left\|S(t - r)^*\right\|_{\mathcal{L}(H)} \left\|\tilde{h}\right\|_H \nonumber\\
&\leq M e^{\omega (t - r)} \left\|\tilde{h}\right\|_H.
\end{align}
Since $\left(Q^{1/2}\right)^*: H \to U$ is bounded (adjoint of 
a Hilbert-Schmidt operator), with 
\begin{equation}
\left\|\left(Q^{1/2}\right)^*\right\|_{\mathcal{L}(H,U)} \leq \left\|Q^{1/2}\right\|_{\mathcal{L}_{\mathrm{HS}}(U,H)},\nonumber
\end{equation}
 we have
\begin{align}
\left\|v_h(r)\right\|_U &= \left\|\left(Q^{1/2}\right)^* S(t - r)^* \tilde{h}\right\|_U \nonumber\\
&\leq \left\|\left(Q^{1/2}\right)^*\right\|_{\mathcal{L}(H,U)} \left\|S(t - r)^* \tilde{h}\right\|_H \nonumber\\
&\leq \left\|Q^{1/2}\right\|_{\mathcal{L}_{\mathrm{HS}}(U,H)} M e^{\omega (t - r)} \left\|\tilde{h}\right\|_H.
\end{align}
This norm is a deterministic function of $r$, finite for all $r \in [0,t]$, and continuous due to the strong continuity of $S(s)^*$. Since $\left(Q^{1/2}\right)^*$, $S(t - r)^*$, and $\tilde{h}\in H$ are all deterministic, we have that $v_h(r)$ is deterministic.
A deterministic process is trivially adapted to any filtration, including $\{\mathcal{F}_r\}_{r \geq 0}$, because $v_h(r)$ is measurable with respect to the trivial $\sigma$-algebra $\{\emptyset,\Omega\}$, which is contained in $\mathcal{F}_r$ for all $r \geq 0$. Thus, $v_h(r)$ is both deterministic and $\{\mathcal{F}_r\}$-adapted, satisfying the prerequisites for It\^o integration.

Finally, we show that the Skorokhod integral \eqref{Skorokhod} reduces to the It\^o integral \eqref{Ito}. A key result in \citep{nualart2006malliavin} states that if $v_h \in \text{Dom}(\delta)$ and satisfies both $v_h(r)$ being $\{\mathcal{F}_r\}$-adapted and $\mathbb{E} \left[ \int_0^t \|v_h(r)\|_U^2 \, dr \right] < \infty$, then $\delta(v_h)$ coincides with the It\^o integral.
We have established already that $v_h(r)$ is deterministic and thus adapted to $\{\mathcal{F}_r\}$. For square-integrability, compute the expectation
\[
\mathbb{E} \left[ \int_0^t \left\|v_h(r)\right\|_U^2 \, dr \right] = \int_0^t \left\|v_h(r)\right\|_U^2 \, dr,
\]
since $v_h(r)$ is deterministic, reducing the expectation to the deterministic norm squared. Estimating, 
\begin{align}
\left\|v_h(r)\right\|_U^2 &= \left\|\left(Q^{1/2}\right)^* S(t - r)^* \tilde{h}\right\|_U^2 \nonumber\\
&\leq \left\|Q^{1/2}\right\|_{\mathcal{L}_{\mathrm{HS}}(U,H)}^2 \left\|S(t - r)^* \tilde{h}\right\|_H^2\nonumber \\
&= \left\|Q^{1/2}\right\|_{\mathcal{L}_{\mathrm{HS}}(U,H)}^2 M^2 e^{2\omega (t - r)} \|\tilde{h}\|_H^2.
\end{align}
Integrating over $[0,t]$
\[
\int_0^t \left\|v_h(r)\right\|_U^2 \, dr \leq \left\|Q^{1/2}\right\|_{\mathcal{L}_{\mathrm{HS}}(U,H)}^2 M^2 \left\|\tilde{h}\right\|_H^2 \int_0^t e^{2\omega (t - r)} \, dr
\]
and using
\[
\int_0^t e^{2\omega (t - r)} \, dr = \begin{cases} 
\frac{1}{2\omega} (e^{2\omega t} - 1) & \text{if } \omega \neq 0, \\
t & \text{if } \omega = 0,
\end{cases}
\]
we obtain 
\[
\int_0^t \left\|v_h(r)\right\|_U^2 \, dr \leq \left\|Q^{1/2}\right\|_{\mathcal{L}_{\mathrm{HS}}(U,H)}^2 M^2 \left\|\tilde{h}\right\|_H^2 \cdot \begin{cases} 
\frac{1}{2|\omega|} |e^{2\omega t} - 1| & \text{if } \omega \neq 0, \\
t & \text{if } \omega = 0,
\end{cases}
\]
which is finite, as $\left\|Q^{1/2}\right\|_{\mathcal{L}_{\mathrm{HS}}(U,H)}$, $M$, $\left\|\tilde{h}\right\|_H$, and $t$ are all finite. Thus, $v_h \in L^2([0,t] \times \Omega; U)$, satisfying the integrability condition.
Since $v_h(r)$ is adapted and square-integrable, the Skorokhod integral reduces to the It\^o integral \eqref{Ito} (or \eqref{Ito1}). Next, recall that the cylindrical Wiener process $W_t$ admits the representation \eqref{F1}, where 
$\{B^k_t\}_{k=1}^{\infty}$ are independent Brownian motions. Substituting \eqref{F1} into \eqref{Ito} yields
\[
\int_0^t \langle v_h(r) , dW_r \rangle_U   = \sum_{n=1}^\infty \int_0^t \langle v_h(r), f_n \rangle_U\, dB^n_r.
\]
This integral is well-defined if 
\[
\sum_{n=1}^\infty \int_0^t \langle v_h(r), f_n \rangle_U^2 \, dr < \infty.
\]
Recalling the definition of $v_h(r)$ \eqref{defC} 
\[
\sum_{n=1}^\infty \left<  \left(Q^{1/2}\right)^* S(t - r)^* \tilde{h}, f_n \right>_U^2 = \left\| \left(Q^{1/2}\right)^* S(t - r)^* \tilde{h} \right\|_U^2,
\]
which we have shown to be integrable. Thus,
\[
\delta(v_h) = \int_0^t \left< \left(Q^{1/2}\right)^* S(t - r)^* \tilde{h}, dW_r \right>_U,
\]
where the integral is well-defined as an It\^o integral, with each component deterministic, confirming the reduction. This completes the proof.

\end{proof}

\subsection{Logarithmic Derivative (Score Function)}

In this section, we derive a closed-form expression for the logarithmic derivative associated with the solution of linear SPDEs of the form \eqref{eq:spde}, driven by space--time coloured noise. In particular, we leverage the Bismut formula discussed in Section \ref{sec:Bismut} and Theorem \ref{thm:SkorokhodIto} to express the logarithmic derivative in terms of the Malliavin covariance operator and the first variation process $Y_t=S(t)$.

\begin{theorem}[Logarithmic derivative / Score function]
\label{thm:score}
Consider the linear SPDE \eqref{eq:spde} with the assumptions at the beginning of Section \ref{sec:methodology}, and let the Malliavin covariance operator $\gamma_t$ be injective\footnote{For instance, $\gamma_t$ is injective if the closed linear span of $\bigcup_{0\le s\le t}\operatorname{Ran}(S(s)Q^{1/2})$ equals $H$ (equivalently, the controllability Gramian has trivial kernel).}. 
Let $\mu_t = \mathcal{N}(S(t) u_0, \gamma_t)$ denote the Gaussian law of $u(t)$ on $H$.
Then for each $h \in \Ran(\gamma_t)$, the logarithmic derivative of $\mu_t$ along $h$ is given by
\begin{equation}
\beta_h(u) = -\langle u - S(t) u_0, \gamma_t^{-1} h \rangle_H,\quad h \in \Ran(\gamma_t),
\label{eq:score-Ran}
\end{equation}
where $\gamma_t^{-1}: \Ran(\gamma_t) \to H$ is the (unbounded, densely defined) inverse of $\gamma_t$ on its range.

The map $h \mapsto \beta_h$ is continuous from $(\CMspace, \|\cdot\|_{\CMspace})$ into $L^2(\mu_t)$. By density of $\Ran(\gamma_t)$ in $\CMspace = \Ran(\gamma_t^{1/2})$, the logarithmic derivative extends uniquely to all $h \in \CMspace$ as an element of $L^2(\mu_t)$, given by the eigenseries formula \eqref{eq:score-series}.
\end{theorem}

\begin{proof}
We derive the logarithmic derivative using the Bismut--Elworthy--Li formula within the infinite-dimensional Malliavin calculus framework, adapting the methodology to the coloured noise case driven by $Q^{1/2} dW_t$.

In the proof of Theorem~\ref{thm:malliavin-covariance-spde} we established that the Malliavin derivative of the solution $u(t)$ with respect to the noise at time $r \in [0, t]$ is
\[
\Mall_r u(t) = S(t - r) Q^{1/2} \mathbf{1}_{[0,t]}(r),
\]
where $\mathbf{1}_{[0,t]}(r)$ is the indicator function ensuring causality. Here $\Mall_r u(t) \in \mathcal{L}_{\mathrm{HS}}(U,H)$. 
The Malliavin covariance operator $\gamma_t$ (see Eq. \eqref{MC}) is positive, self-adjoint, and trace-class under the stated assumptions on $Q^{1/2}$ and $S(t)$, as it arises from the integral of positive semi-definite operators in the context of SPDEs with additive noise. The injectivity assumption on $\gamma_t$ (equivalently, $\ker \gamma_t = \{0\}$) ensures that all finite-dimensional projections $\Pi u(t)$ have non-degenerate covariance matrices, and hence smooth densities with respect to Lebesgue measure on the range of $\Pi$. Since $\gamma_t$ is trace-class (hence compact), the inverse $\gamma_t^{-1}$ is \emph{unbounded} as an operator on $H$, even when $\gamma_t$ is injective. Injectivity together with the trace-class property implies that $\gamma_t$ has dense range (i.e., $\overline{\Ran(\gamma_t)} = H$), but the eigenvalues of $\gamma_t$ accumulate at zero, precluding any uniform positive lower bound. This is why $\gamma_t^{-1}: \Ran(\gamma_t) \to H$ is densely defined but unbounded. For linear SPDEs with additive noise, Malliavin differentiability of $u(t)$ follows directly from the explicit representation \eqref{sol}, and no Hörmander-type condition is needed. The injectivity of $\gamma_t$ is equivalent to the approximate controllability condition: the closed linear span of $\bigcup_{0 \le s \le t} \Ran(S(s)Q^{1/2})$ equals $H$. This ensures that noise propagates to all directions in $H$.
For a direction $h \in \Ran(\gamma_t)$, let $\tilde{h} = \gamma_t^{-1} h \in H$. The covering vector field (Definition~\ref{def:covering-field}) is
\[
v_h(r) = \left(Q^{1/2}\right)^* S(t - r)^* \tilde{h} \, \mathbf{1}_{[0,t]}(r),
\]
which is a $U$-valued, deterministic process. By Theorem~\ref{thm:covering-property-spde}, this satisfies the covering property $\langle \Mall u(t), v_h \rangle_{H_W} = h$.

Since $v_h(r)$ is deterministic and adapted, as shown in Theorem~\ref{thm:SkorokhodIto}, the Skorokhod integral coincides with the It\^o integral. The mild solution representation \eqref{sol} gives
\begin{equation}
u(t) - S(t) u_0 = \int_0^t S(t - s) Q^{1/2} \, dW_s =: z,
\label{SI}
\end{equation}
where $z$ is a centred Gaussian random variable in $H$ with covariance $\gamma_t$.

For any $\xi \in H$, we have
\begin{align}
\langle z, \xi \rangle_H &= \left\langle \int_0^t S(t - s) Q^{1/2} \, dW_s, \xi \right\rangle_H \nonumber\\ 
&= \int_0^t \left< \left(Q^{1/2}\right)^* S(t - s)^* \xi, dW_s \right>_U. 
\end{align}
Taking $\xi = \tilde{h} = \gamma_t^{-1} h$:
\begin{equation}
\langle z, \tilde{h} \rangle_H = \int_0^t \left<\left(Q^{1/2}\right)^* S(t - s)^* \tilde{h}, dW_s \right>_U = \delta(v_h).
\label{eq:z-skorokhod}
\end{equation}

By the Bismut formula \eqref{scoreBismut}, defining $G := \mathbb{E}[\delta(v_h) \mid \sigma(u(t))]$, we have $\beta_h(u(t)) = -G$ as $\sigma(u(t))$-measurable random variables. Since $\delta(v_h) = \langle z, \tilde{h} \rangle_H$ (by \eqref{eq:z-skorokhod}) and $z = u(t) - S(t)u_0$, the conditional expectation simplifies:
\[
G = \mathbb{E}[\langle u(t) - S(t)u_0, \tilde{h} \rangle_H \mid \sigma(u(t))] = \langle u(t) - S(t)u_0, \tilde{h} \rangle_H,
\]
since $\langle u(t) - S(t)u_0, \tilde{h} \rangle_H$ is already $\sigma(u(t))$-measurable. Thus
\[
\beta_h(u(t)) = -\langle u(t) - S(t) u_0, \gamma_t^{-1} h \rangle_H \quad \mu_t\text{-a.e.},
\]
which is formula \eqref{eq:score-Ran}.

For the extension to all $h \in \CMspace = \Ran(\gamma_t^{1/2})$: since $\Ran(\gamma_t)$ is dense in $\CMspace$ with respect to the Cameron--Martin norm, we extend by continuity in $L^2(\mu_t)$. Let $\{e_k\}_{k=1}^\infty$ be an orthonormal basis of $H$ consisting of eigenvectors of $\gamma_t$, with $\gamma_t e_k = \lambda_k e_k$ and $\lambda_k > 0$ (by injectivity). For $h \in \CMspace$, write $h = \sum_{k=1}^\infty h_k e_k$ where $\sum_{k=1}^\infty h_k^2/\lambda_k < \infty$ (the Cameron--Martin condition). The logarithmic derivative is
\begin{equation}
\beta_h(u) = -\sum_{k=1}^\infty \frac{h_k}{\lambda_k} \langle u - S(t)u_0, e_k \rangle_H,
\label{eq:score-series}
\end{equation}
where the series converges in $L^2(\mu_t)$.

\end{proof}

In finite dimensions where $H = \mathbb{R}^n$, the covariance matrix $\gamma_t$ is positive definite (under our injectivity assumption), so $\gamma_t^{-1} \in \mathcal{L}(\mathbb{R}^n)$ is a bounded operator and $\CMspace = \Ran(\gamma_t) = \mathbb{R}^n$. In this case, formula \eqref{eq:score-Ran} can be written equivalently as
\[
\nabla \log p_{u(t)}(u) = -\gamma_t^{-1}(u - S(t)u_0),
\]
which is the standard score function for a Gaussian distribution with mean $S(t)u_0$ and covariance $\gamma_t$. In infinite dimensions, such a formula is \emph{not} well-posed because $\gamma_t^{-1}$ is unbounded and $(u - S(t)u_0) \notin \Ran(\gamma_t)$ generically. The correct formulation is \eqref{eq:score-Ran}, where the unbounded inverse $\gamma_t^{-1}$ acts on the direction $h \in \Ran(\gamma_t)$, not on the random state.

One might ask: if $u(t)$ is Gaussian, why employ Malliavin calculus? The answer is twofold. First, the Bismut formula expresses the score as a conditional expectation $\beta_h = -\mathbb{E}[\delta(v_h) \mid \sigma(u(t))]$, which provides an intrinsic representation that does not rely on the explicit Gaussian form; this representation extends naturally to settings where the law of $u(t)$ is non-Gaussian. Second, working with Cameron--Martin spaces and the eigenseries representation \eqref{eq:score-series} ensures an intrinsic infinite-dimensional formulation that is mathematically rigorous and independent of any discretisation scheme. The Malliavin/Bismut machinery presented here could serve as a starting point for extensions to nonlinear SPDEs with multiplicative noise, where the law of $u(t)$ is no longer Gaussian and the score becomes a genuine conditional expectation; such extensions, however, remain the subject of future work.

\section{Numerical Results}
\label{sec:experiments}

In this section, we validate the closed-form Malliavin score formula
\begin{equation}
\beta_h(u) = -\langle u - S(t) u_0, \gamma_t^{-1} h \rangle_H, \quad h \in \Ran(\gamma_t),
\label{eq:score-validation}
\end{equation}
we obtained in Theorem~\ref{thm:score} through numerical simulations of SPDEs in one- and two-dimensional spatial domains. Specifically, in Section~\ref{sec:1d-galerkin}, we consider one-dimensional SPDEs on a bounded domain with Dirichlet boundary conditions, discretised via spectral Galerkin truncation in the Laplacian eigenbasis. In Section~\ref{sec:2d-spectral}, we extend the validation to two-dimensional SPDEs on a periodic domain, discretised using a Fourier spectral method. These experiments confirm that the score formula applies to both second- and fourth-order operators, independently of the spatial dimension, boundary conditions, and choice of spectral basis.

\subsection{One Dimension}
\label{sec:1d-galerkin}

 All numerical experiments in this section are conducted on one-dimensional linear SPDEs of the form~\eqref{eq:spde} defined on 
 the spatial domain $(0,1)$ with homogeneous Dirichlet boundary conditions. We use second-order finite-differences to approximate the directional derivative of the log-density, against which we validate the closed-form formula~\eqref{eq:score-validation}.

\subsubsection{Galerkin discretisation}

We consider the state space $H = L^2(0,1)$ with homogeneous Dirichlet boundary conditions, for which the Laplacian $\Delta$ admits the orthonormal eigenbasis $\{\varphi_k\}_{k=1}^\infty$ given by
\begin{equation}
\varphi_k(x) = \sqrt{2} \sin(k\pi x), \quad -\Delta \varphi_k = \lambda_k \varphi_k, \quad \lambda_k = (k\pi)^2.
\label{eq:eigenbasis}
\end{equation}
For operators $A$ that are functions of the Laplacian, i.e., $A = f(-\Delta)$ for some function $f: \mathbb{R}_+ \to \mathbb{R}$, the eigenbasis~\eqref{eq:eigenbasis} diagonalises $A$ with eigenvalues $a_k = f(\lambda_k)$.

The spectral Galerkin approximation truncates the expansion to $N$ modes. In eigenspace coordinates, the SPDE~\eqref{eq:spde} decouples into $N$ independent scalar Ornstein--Uhlenbeck processes:
\begin{equation}
du_k(t) = a_k u_k(t) \, dt + \sqrt{q_k} \, dB_k(t), \quad k = 1, \ldots, N,
\label{eq:mode-dynamics}
\end{equation}
where $u_k(t) = \langle u(t), \varphi_k \rangle_H$ are the Fourier coefficients, $\{B_k(t)\}_{k=1}^N$ are independent standard Brownian motions, and $q_k$ are the eigenvalues of the noise covariance operator $Q$. We take $q_k = k^{-2}$, ensuring trace-class regularity $\sum_{k=1}^\infty q_k < \infty$.
The rapid algebraic decay of $q_k$ ensures that the Galerkin truncation to $N = 64$ modes fully resolves the solution. The total solution variance $\mathrm{Tr}(\gamma_t)$ is identical at $N = 64$ and $N = 128$ to six significant figures, with the residual noise variance in modes $k > 64$ accounting for less than $0.5\%$ of the total. Increasing the truncation level produces no visible change in the solution fields or the score.
The mild solution admits the representation
\begin{equation}
u_k(t) = e^{a_k t} u_{0,k} + \int_0^t e^{a_k(t-s)} \sqrt{q_k} \, dB_k(s),
\label{eq:mild-eigenspace}
\end{equation}
which is Gaussian with mean $m_k(t) = e^{a_k t} u_{0,k}$, where $u_{0,k} = \langle u_0, \varphi_k \rangle_H$, and variance
\begin{equation}
\gamma_k(t) = q_k  \frac{e^{2a_k t} - 1}{2a_k},
\label{eq:gamma-eigenspace}
\end{equation}
with the convention $\gamma_k(t) = q_k t$ when $a_k = 0$. The Malliavin covariance operator $\gamma_t$ is diagonal in the eigenbasis with eigenvalues~\eqref{eq:gamma-eigenspace}.

\subsubsection{1D SPDE classes}
\label{sec:SPDEclasses}
We validate the score formula across four classes of linear SPDEs. Table~\ref{tab:spde-classes} summarises the operator structure and eigenvalues for each class.
\begin{table}[t]
\centering
\begin{tabular}{@{}llll@{}}
\toprule
\textbf{SPDE} & \textbf{Operator $A$} & \textbf{Eigenvalues $a_k$} & \textbf{Parameters} \\
\midrule
Heat Equation & $\Delta$ & $-(k\pi)^2$ & --- \\
Ornstein--Uhlenbeck & $\Delta - \kappa I$ & $-(k\pi)^2 - \kappa$ & $\kappa = 2$ \\
Scaled Diffusion & $\nu \Delta$ & $-\nu(k\pi)^2$ & $\nu = 0.1$ \\
Fractional Laplacian & $-(-\Delta)^\alpha$ & $-(k\pi)^{2\alpha}$ & $\alpha = 0.75$ \\
\bottomrule
\end{tabular}
\caption{One-dimensional SPDE classes used for numerical validation of the Malliavin score formula~\eqref{eq:score-validation}. All operators are of the form $A = f(-\Delta)$, ensuring diagonalisation in the sine eigenbasis. Stability requires $a_k < 0$ for all $k \geq 1$.}
\label{tab:spde-classes}
\end{table}
The first three operators correspond to classical diffusion processes: Brownian diffusion (heat equation); the Ornstein--Uhlenbeck process (heat equation with linear damping leading to faster relaxation towards equilibrium); and scaled Brownian diffusion (heat equation with reduced diffusivity). The fractional Laplacian with $\alpha=0.75$ (order $2\alpha=1.5$) models anomalous subdiffusion.

\subsubsection{Validation methodology}

For each SPDE class, we validate the Malliavin score formula~\eqref{eq:score-validation} against a central finite-difference approximation of the directional derivative. In the $N$-dimensional Galerkin truncation, the projected solution $u^{(N)}(t) \in \mathbb{R}^N$ admits a Lebesgue density $p_t^{(N)}$. For a direction $h \in H$ with $\|h\|_H = 1$, the finite-difference approximation is
\begin{equation}
\beta_h^{\mathrm{FD}}(u) = \frac{\log p_t^{(N)}(u + \varepsilon h) - \log p_t^{(N)}(u - \varepsilon h)}{2\varepsilon},
\label{eq:fd-approx}
\end{equation}
with step size $\varepsilon = 10^{-5}$. Since $u^{(N)}(t) \sim \mathcal{N}(m_t, \gamma_t)$ is Gaussian, the log-density (up to an additive constant) is
\begin{equation}
\log p_t^{(N)}(u) = -\frac{1}{2} \sum_{k=1}^N \frac{(u_k - m_k)^2}{\gamma_k},
\label{eq:log-density}
\end{equation}
and the finite-difference approximation~\eqref{eq:fd-approx} can be computed exactly in eigenspace.
The Malliavin score formula in eigenspace coordinates reads
\begin{equation}
\beta_h(u) = -\sum_{k=1}^N \frac{(u_k - m_k) h_k}{\gamma_k},
\label{eq:malliavin-eigenspace}
\end{equation}
where $h_k = \langle h, \varphi_k \rangle_H$ are the Fourier coefficients of the direction $h$. We take $h = \varphi_1$ (the first eigenmode) throughout, corresponding to the lowest-frequency perturbation.

For each SPDE class listed in Table~\ref{tab:spde-classes}, we simulate $P = 4$ independent sample paths of the solution process $\{u(t)\}_{t \in [0,T]}$ with $T = 1$ and initial condition $u_0 = \varphi_1 + \tfrac{1}{2}\varphi_2 + \tfrac{1}{4}\varphi_3$. The truncation level is $N = 64$ modes, and we evaluate the score at $50$ uniformly spaced time points $t_j \in [0.02, 1]$. At each time point and along each sample path, we compute:
\begin{enumerate}
    \item The Malliavin score $\beta_h(u(t_j))$ via formula~\eqref{eq:malliavin-eigenspace};
    \item The finite-difference score $\beta_h^{\mathrm{FD}}(u(t_j))$ via formula~\eqref{eq:fd-approx}.
\end{enumerate}
The absolute error $|\beta_h(u) - \beta_h^{\mathrm{FD}}(u)|$ quantifies the discrepancy between the two methods.

\subsubsection{Results}

Figure~\ref{fig:validation-second} presents the numerical validation results for the four SPDE classes discussed in Section~\ref{sec:SPDEclasses}.
\begin{figure}[t]
\centering
\includegraphics[width=\textwidth]{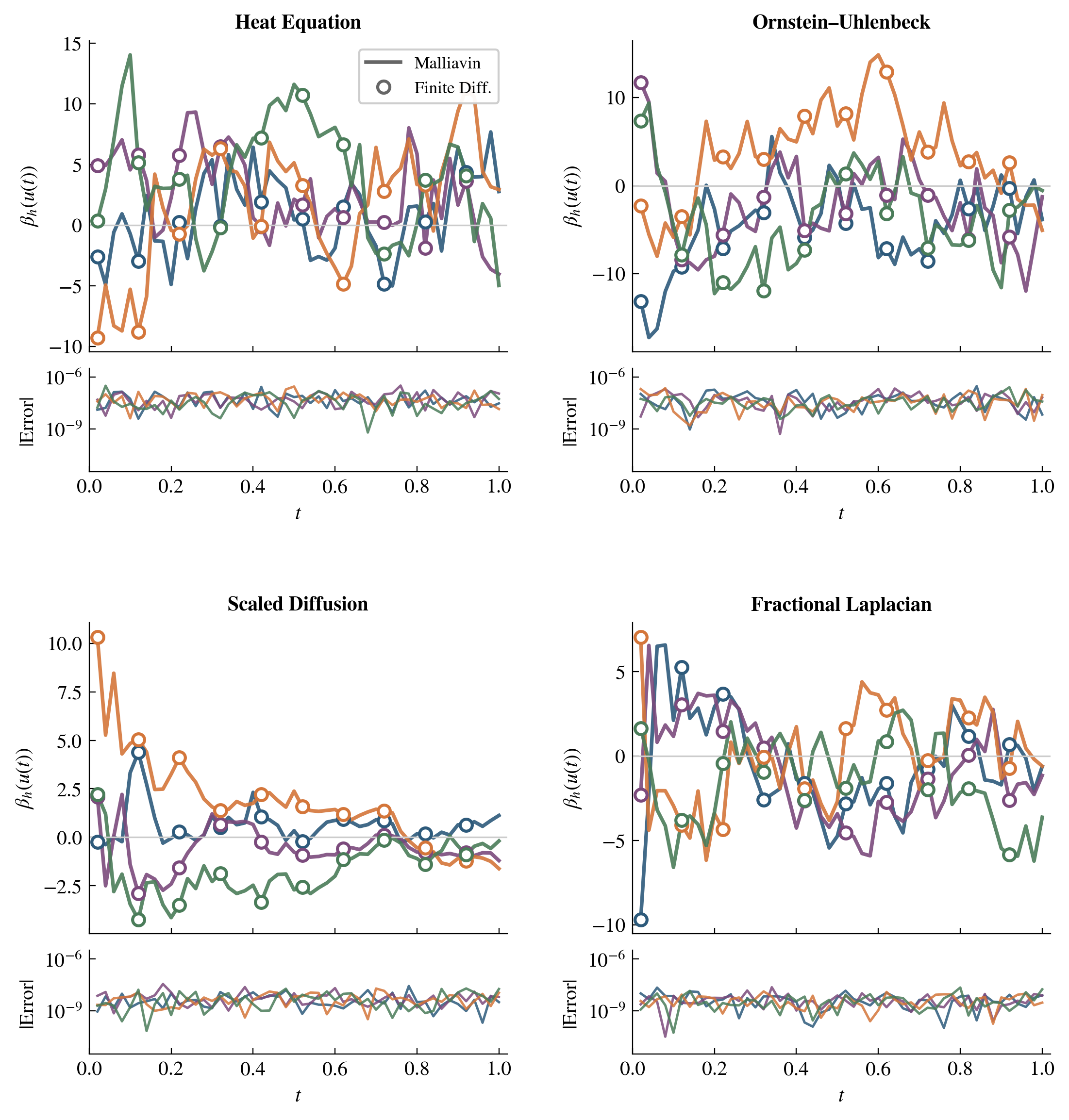}
\caption{Numerical validation of the Malliavin score formula~\eqref{eq:malliavin-eigenspace} for four classes of linear SPDEs (see Table~\ref{tab:spde-classes}). \emph{Upper panels}: Score trajectories $\beta_h(u(t))$ along four sample paths, comparing the Malliavin formula (solid lines) with finite-difference approximation (hollow circles). \emph{Lower panels}: Absolute error $|\beta_h - \beta_h^{\mathrm{FD}}|$ on logarithmic scale. The dashed line indicates machine precision. All four SPDEs achieve errors of $\mathcal{O}(10^{-8})$, consistent with central-difference truncation error.}
\label{fig:validation-second}
\end{figure}
It is seen that for the heat equation, Ornstein--Uhlenbeck, scaled diffusion, and fractional Laplacian SPDEs, the Malliavin and finite-difference scores exhibit excellent agreement, with errors in the range $10^{-11}$ to $10^{-7}$. This is consistent with the expected truncation error of the central finite-difference scheme: for a function $f$ with bounded fourth derivative, $|f'(x) - (f(x+\varepsilon) - f(x-\varepsilon))/(2\varepsilon)| = \mathcal{O}(\varepsilon^2)$, giving $\mathcal{O}(10^{-10})$ for $\varepsilon = 10^{-5}$. The slightly larger observed errors arise from amplification by the curvature of the log-density, which scales with $\gamma_k^{-1}$.
These results demonstrate that Theorem~\ref{thm:score} applies to any linear SPDE of the form~\eqref{eq:spde} where $A$ is diagonalisable in a common eigenbasis with the noise covariance $Q$. The specific physics encoded in the operator $A$, whether diffusion, damping, or higher-order smoothing, enters only through the eigenvalues $a_k$, which determine the Malliavin covariance~\eqref{eq:gamma-eigenspace}. The score formula itself is structurally identical across all cases.

\subsection{Two Dimensions}
\label{sec:2d-spectral}

The one-dimensional experiments of Section~\ref{sec:1d-galerkin} employed spectral Galerkin discretisation in the sine eigenbasis, which diagonalises the operator $A$ by construction. To confirm that the Malliavin score formula applies independently of the spatial dimension, boundary conditions, and choice of spectral basis, we now validate it on the two-dimensional periodic domain $[0,2\pi]^2$ using a Fourier spectral method with both second- and fourth-order operators.

\subsubsection{Fourier spectral discretisation}

We consider the uniform grid on $[0,2\pi]$ given by $x_j = 2\pi j/(N+1)$ for $j=0,\ldots,N$ with $N=48$, and construct the associated two-dimensional tensor-product grid on $[0,2\pi]\times [0,2\pi]$, yielding $49 \times 49 = 2{,}401$ degrees of freedom. Since the operators in Table~\ref{tab:spde-classes-2d} depend only on $|k|^2 = k_1^2 + k_2^2$, they are diagonal in the Fourier basis, with analytically known eigenvalues $a_{k_1,k_2}$. All computations---including eigenvalue evaluation, covariance construction, sampling, and score evaluation---are performed mode-by-mode via the two-dimensional FFT, resulting in a Fourier spectral (Galerkin) discretisation on the periodic domain.
As an independent cross-validation, we construct the physical-space differentiation matrix $D \in \mathbb{R}^{(N+1) \times (N+1)}$ for the odd-point trigonometric interpolant
\[
D_{ij} = \begin{cases} \displaystyle\frac{(-1)^{i+j}}{2\sin\bigl((x_i - x_j)/2\bigr)}, & i \neq j, \\[6pt] 0, & i = j, \end{cases}
\]
and assemble the two-dimensional Laplacian and biharmonic operators via Kronecker products
\[
\Delta_{2\mathrm{D}} = D^{(2)} \otimes I + I \otimes D^{(2)}, \qquad \Delta_{2\mathrm{D}}^2 = D^{(4)} \otimes I + 2\, D^{(2)} \otimes D^{(2)} + I \otimes D^{(4)},
\]
where $D^{(2)} = D^2$ and $D^{(4)} = (D^2)^2$. The numerically computed spectrum of these matrices agrees with the analytical eigenvalues to within $\mathcal{O}(10^{-4})$, confirming the correctness of the spectral implementation.

\subsubsection{2D SPDE classes}

We validate the score formula across four linear SPDEs on $[0,2\pi]^2$ with periodic boundary conditions, encompassing both second- and fourth-order operators.

\begin{table}[t]
\centering
\begin{tabular}{@{}llll@{}}
\toprule
\textbf{SPDE} & \textbf{Operator $A$} & \textbf{Eigenvalues $a_{k_1,k_2}$} & \textbf{Parameters} \\
\midrule
Stochastic Heat Equation & $\nu\Delta$ & $-\nu(k_1^2 + k_2^2)$ & $\nu = 1$ \\
Ornstein--Uhlenbeck & $\nu\Delta - \kappa I$ & $-\nu(k_1^2 + k_2^2) - \kappa$ & $\nu = 1,\; \kappa = 2$ \\
Stochastic Biharmonic & $-\mu\Delta^2$ & $-\mu(k_1^2 + k_2^2)^2$ & $\mu = 1$ \\
Swift--Hohenberg & $(r-1)I - 2\Delta - \Delta^2$ & $r - (1 - k_1^2 - k_2^2)^2$ & $r = -0.5$ \\
\bottomrule
\end{tabular}
\caption{Two-dimensional SPDE classes used for the score validation. The first two operators are second-order; the last two are fourth-order.}
\label{tab:spde-classes-2d}
\end{table}

The stochastic heat equation and Ornstein--Uhlenbeck process are second-order operators representing classical diffusion. The stochastic biharmonic equation is fourth-order and models surface diffusion and thin-film dynamics. The linearised Swift--Hohenberg equation combines second- and fourth-order terms and arises in pattern formation. Noise covariance eigenvalues are $q_{k_1,k_2} = (1 + k_1^2 + k_2^2)^{-2}$, which is trace-class in two dimensions.

\subsubsection{Spectral properties of the forcing term and resolution study}

Before validating the score formula, we examine the spectral properties of the coloured noise to confirm that the Fourier truncation to $N+1=49$ modes per direction adequately resolves the forcing. In Figure~\ref{fig:noise-2d} we plot four independent realisations of the noise field $Q^{1/2}\xi(x,y)$ at a fixed time. The samples exhibit the characteristic structure of coloured noise: random spatial patterns with $\mathcal{O}(1)$ amplitude and a smooth appearance, reflecting the rapid spectral decay $q_{k_1,k_2} = (1 + |k|^2)^{-2} \sim |k|^{-4}$ at high wavenumbers.
The PDE solution is smoother still. For the stochastic heat equation, the semigroup $S(t) = e^{t\Delta}$ further damps mode $k$ by a factor $e^{-\nu|k|^2 t}$, so the per-mode variance of the stochastic convolution decays as $\gamma_{k}(t) \sim |k|^{-2s-2}$ for large $|k|$, where $s$ is the exponent in the noise covariance $q_k = (1 + |k|^2)^{-s}$ (here $s = 2$). For the fourth-order operators the decay is $|k|^{-2s-4}$. Thus, the coloured noise that ensures well-posedness of the SPDE also acts as a spectral filter, and the semigroup provides further smoothing.
To verify that the Fourier truncation is sufficient, we compare the total solution variance $\mathrm{Tr}(\gamma_t)$ at two resolutions, $49 \times 49$ modes (used in our experiments) and $99 \times 99$ modes. For the stochastic heat equation at $t = 1$, the values are $\mathrm{Tr}(\gamma_t) = 1.597227$ and $1.597228$, respectively---identical to six significant figures. The variance carried by modes beyond $|k| > 24$ (i.e., those added by the higher resolution) accounts for less than $0.0001\%$ of the total. Likewise, the noise variance beyond $|k| > 24$ is less than $0.1\%$ of the total. These results confirm that with the spectral decay $s = 2$, the truncation to $49$ modes per direction fully resolves both the forcing and the solution, and increasing the resolution produces no change in the computed fields or score errors.

\begin{figure}[htbp]
\centering
\includegraphics[width=\textwidth]{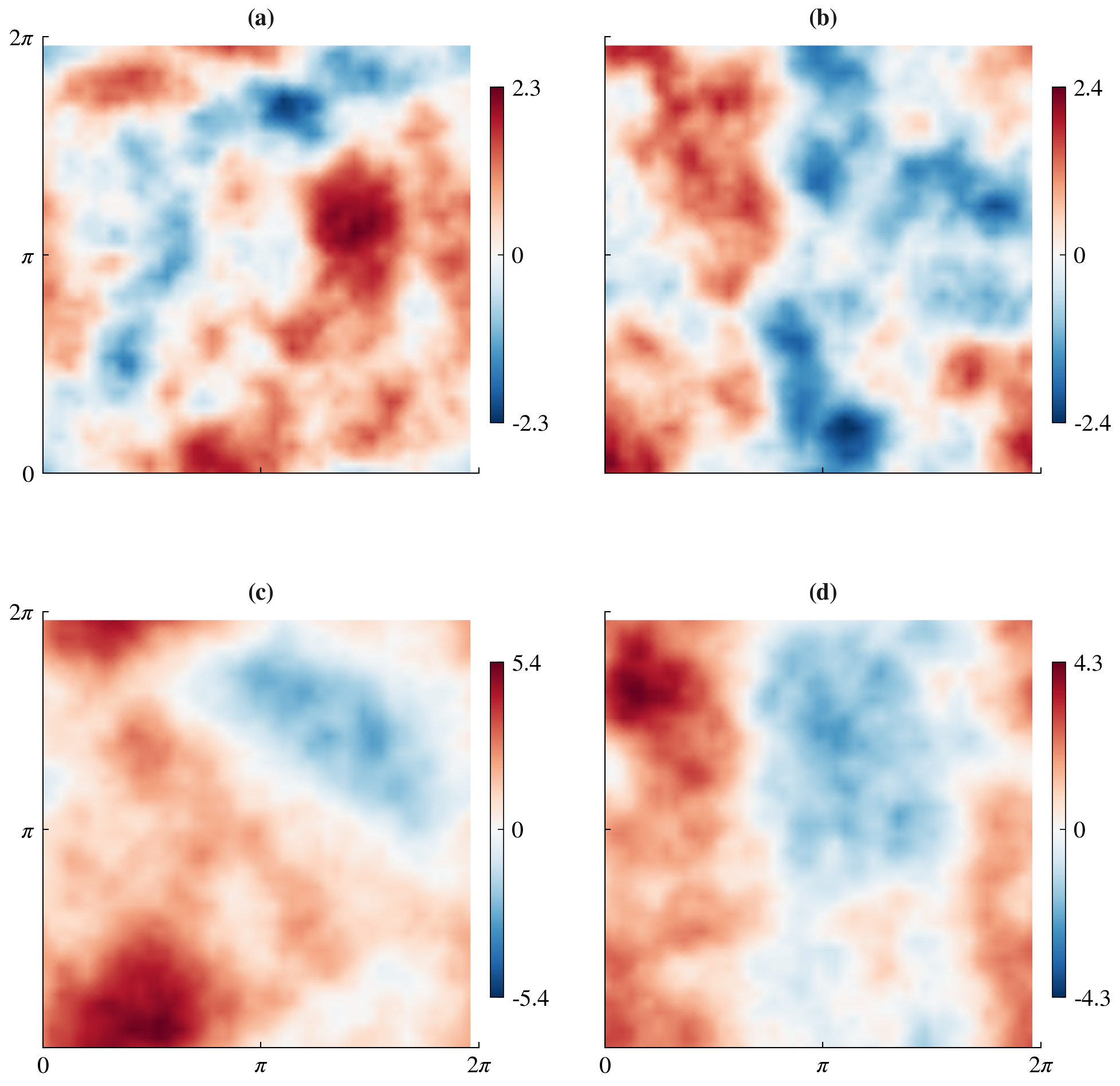}
\caption{Four independent realisations of the coloured noise field $Q^{1/2}\xi(x,y)$ on $[0,2\pi]^2$, with noise covariance eigenvalues $q_{k_1,k_2} = (1+k_1^2+k_2^2)^{-2}$. The fields have $\mathcal{O}(1)$ amplitude and exhibit stochastic spatial structure, but are smooth due to the rapid spectral decay of the covariance. This is the forcing that drives each of the SPDEs in Table~\ref{tab:spde-classes-2d}; the semigroup $S(t)$ further damps high-frequency modes, producing the smoother solution fields shown in Figure~\ref{fig:solution-2d}.}
\label{fig:noise-2d}
\end{figure}

\subsubsection{Validation methodology}

Since the SPDE~\eqref{eq:spde} is linear with additive Gaussian noise, the solution $u(t)$ is Gaussian with known mean and covariance in Fourier space. Each Fourier mode $\hat{u}_{k_1,k_2}(t)$ is an independent complex Gaussian with mean $$\hat{m}_{k_1,k_2}(t) = e^{a_{k_1,k_2} t}\, \hat{u}_{0,k_1,k_2}$$ and variance $$\gamma_{k_1,k_2}(t) = q_{k_1,k_2}(e^{2a_{k_1,k_2} t} - 1)/(2a_{k_1,k_2}),$$ where $\hat{u}_{0,k_1,k_2}$ denotes the Fourier coefficient of the initial condition. We take $u_0(x,y) = \cos x + \tfrac{1}{2}\cos y + \tfrac{1}{4}\cos(x+y)$ and evaluate at $t = 1$. The modes are sampled analytically via the FFT, avoiding time-stepping errors entirely.
For each SPDE class, we compute two quantities at each grid point $(x_i, y_j)$:
\begin{enumerate}
\item The Malliavin score from Theorem~\ref{thm:score}. Note that in Fourier space we have $\hat{s}_{k_1,k_2} = -(\hat{u}_{k_1,k_2} - \hat{m}_{k_1,k_2})/\gamma_{k_1,k_2}$. This can be mapped to physical space via the inverse FFT.
\item A per-mode finite-difference approximation of the score in Fourier space. For each mode $(k_1,k_2)$, the log-density contribution is $-|\hat{c}_{k_1,k_2}|^2/(2\gamma_{k_1,k_2})$, where $\hat{c}_{k_1,k_2} = \hat{u}_{k_1,k_2} - \hat{m}_{k_1,k_2}$. Writing $\hat{c} = \alpha + i\beta$ (suppressing mode indices), the central finite difference is applied to the real and imaginary parts separately
\begin{equation}
\hat{s}^{\mathrm{FD}} = \frac{1}{2\varepsilon}\!\left[\!-\frac{(\alpha{+}\varepsilon)^2 + \beta^2}{2\gamma} + \frac{(\alpha{-}\varepsilon)^2 + \beta^2}{2\gamma}\right] + \frac{i}{2\varepsilon}\!\left[\!-\frac{\alpha^2 + (\beta{+}\varepsilon)^2}{2\gamma} + \frac{\alpha^2 + (\beta{-}\varepsilon)^2}{2\gamma}\right].
\label{eq:fd-2d}
\end{equation}
Using the algebraic identity $(x+\varepsilon)^2 - (x-\varepsilon)^2 = 4\varepsilon x$, which is exact for quadratics, this simplifies to $\hat{s}^{\mathrm{FD}}_{k_1,k_2} = -\hat{c}_{k_1,k_2}/\gamma_{k_1,k_2}$ for any $\varepsilon > 0$, yielding zero truncation error and avoiding catastrophic cancellation in high-frequency modes. The physical-space FD score is then obtained via the inverse FFT of $\hat{s}^{\mathrm{FD}}$.
\end{enumerate}
The pointwise error field $|s_{\mathrm{Mall}}(x,y) - s_{\mathrm{FD}}(x,y)|$ is computed from a single inverse FFT of the Fourier-space difference, ensuring that the physical-space error reflects only the per-mode discrepancy.

\subsubsection{Results}

Figure~\ref{fig:solution-2d} displays the stochastic component $(u - S(t)u_0)(x,y)$ for each of the four SPDEs at $t = 1$. The heat equation and Ornstein--Uhlenbeck fields show mid-frequency fluctuations, while the fourth-order operators (biharmonic, Swift--Hohenberg) produce smoother fields due to stronger damping of high-frequency modes.

The visual smoothness of the stochastic component is a direct consequence of the trace-class noise assumption that underpins the well-posedness of~\eqref{eq:spde}. For the stochastic heat equation, the per-mode variance of the stochastic convolution decays as $\gamma_{k_1,k_2}(t) \sim q_{k_1,k_2} / (2\nu|k|^2) \sim |k|^{-2s-2}$ for large $|k|^2 = k_1^2 + k_2^2$, where $s$ is the exponent in the noise covariance $q_{k_1,k_2} = (1 + |k|^2)^{-s}$. With $s = 2$, this gives $\gamma_k \sim |k|^{-6}$, which by the Sobolev embedding theorem in two dimensions places the sample paths almost surely in $H^{\alpha}$ for $\alpha < s = 2$, and hence in $C^0([0,2\pi]^2)$. The same mechanism applies to all four operators, with the fourth-order cases exhibiting even faster spectral decay. This regularity is characteristic of SPDEs driven by spatially correlated noise with trace-class covariance, where the coloured noise that ensures well-posedness in $H$ simultaneously regularises the solution~\citep{daprato2014stochastic, ferrante2006spdes}. In contrast, space--time white noise ($Q = I$, $s = 0$) does not produce function-valued solutions in $d \geq 2$~\citep{hairer2009introduction}. Thus, the observed smoothness is not a numerical artefact but a genuine feature of the coloured-noise regime studied in this paper.

Figure~\ref{fig:error-2d} shows the corresponding pointwise score error field $|s_{\mathrm{Mall}}(x,y) - s_{\mathrm{FD}}(x,y)|$. The error is spatially unstructured and lies at machine precision throughout: $\mathcal{O}(10^{-10})$ for the second-order operators and $\mathcal{O}(10^{-9})$ for the fourth-order operators, where the larger eigenvalues (${\sim}10^6$) amplify floating-point rounding. In Fourier space, the maximum error across all modes is $\mathcal{O}(10^{-11})$ for second-order and $\mathcal{O}(10^{-10})$ for fourth-order operators. These results confirm that Theorem~\ref{thm:score} applies to arbitrary linear SPDEs of the form~\eqref{eq:spde} on periodic domains in two spatial dimensions, independently of the order of the differential operator and the discretisation method.

\begin{figure}[t]
\centering
\includegraphics[width=\textwidth]{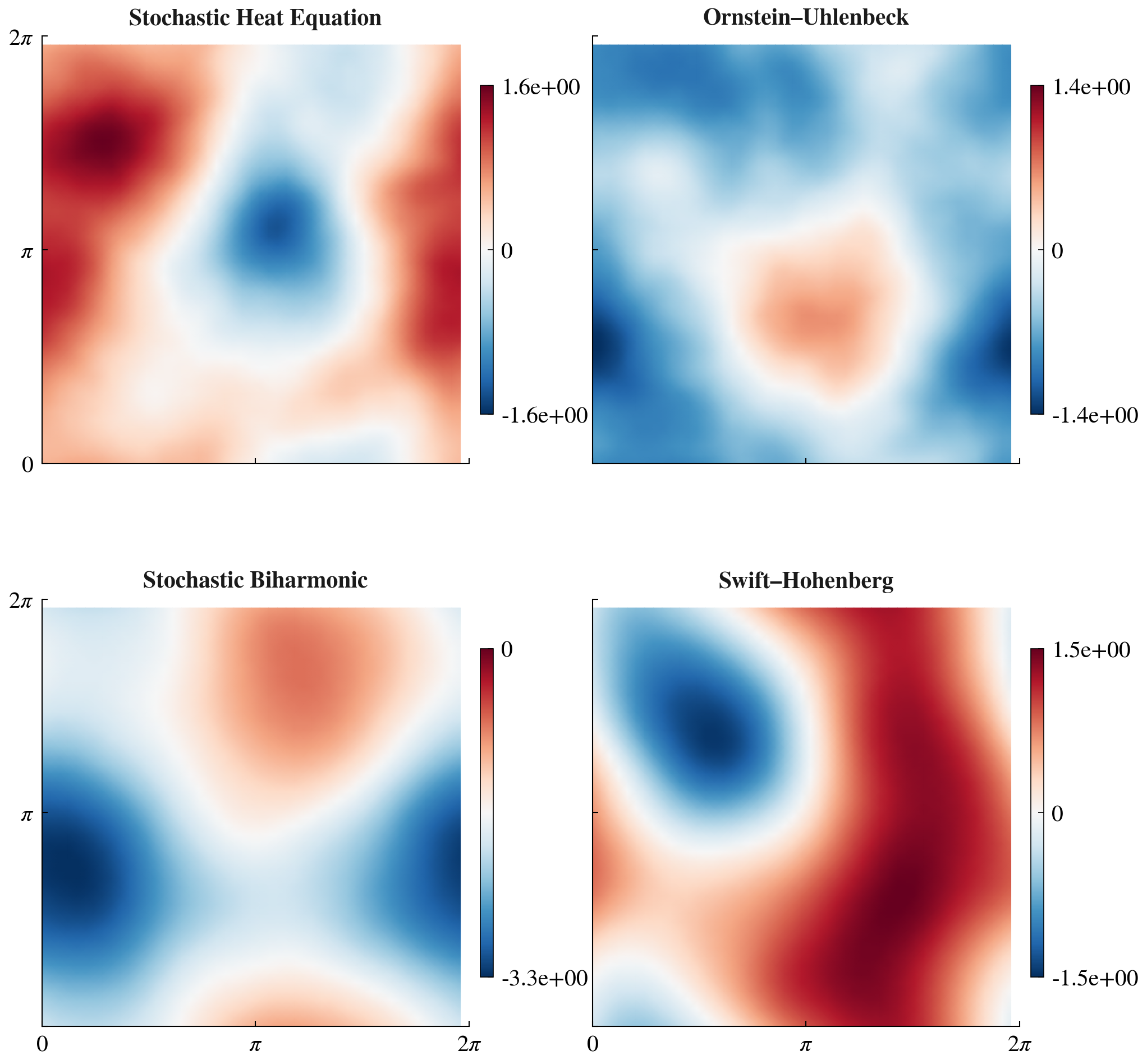}
\caption{Stochastic component $(u - S(t)u_0)(x,y)$ (one sample) at $t = 1$ for the four SPDE classes in Table~\ref{tab:spde-classes-2d}, computed on $[0,2\pi]^2$ with $49 \times 49$ Fourier spectral grid points. The spatial structure reflects the spectral properties of each operator: fourth-order operators produce smoother fields due to stronger high-frequency damping.}
\label{fig:solution-2d}
\end{figure}

\begin{figure}[t]
\centering
\includegraphics[width=\textwidth]{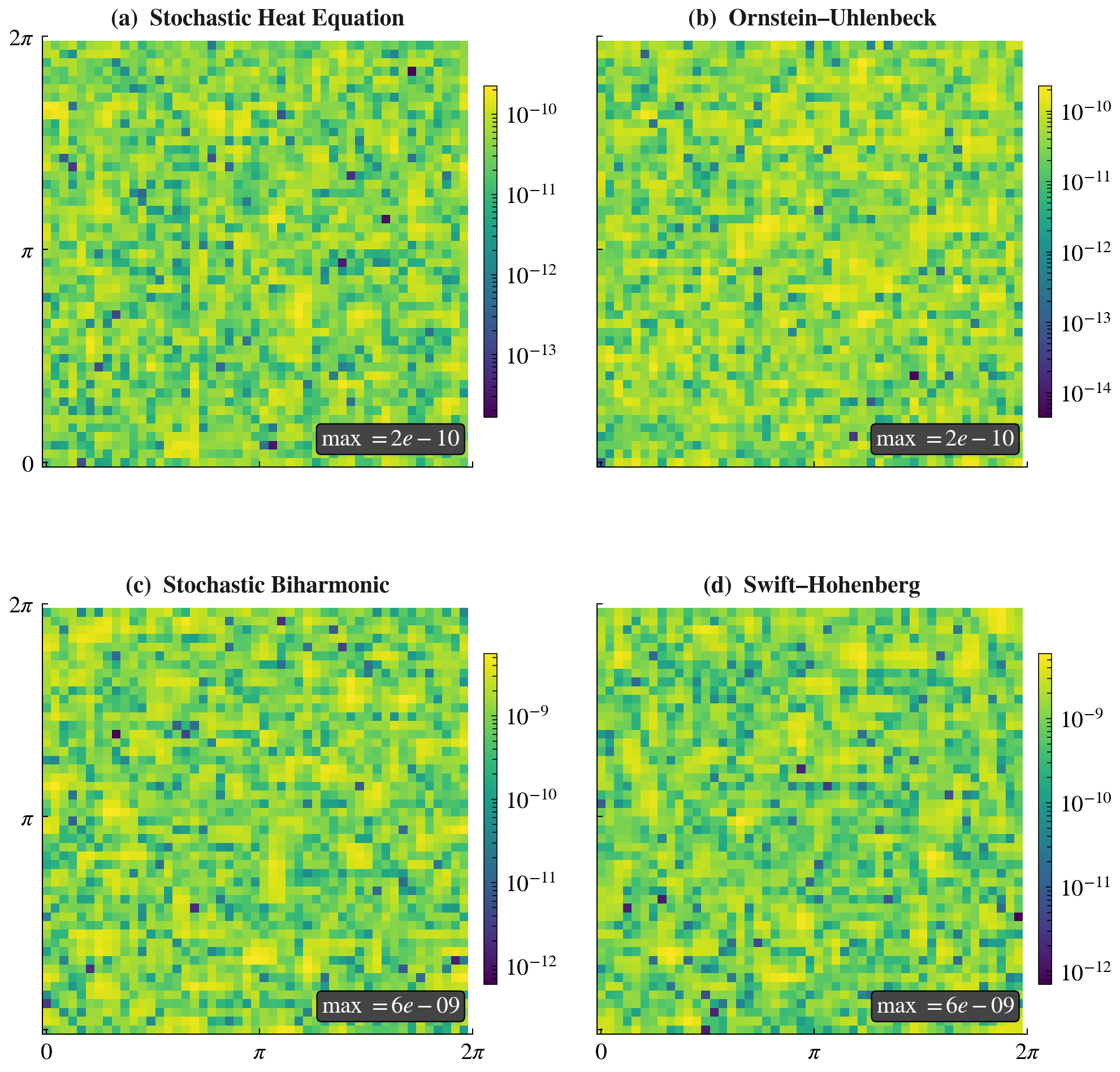}
\caption{Pointwise score error $|s_{\mathrm{Mall}} - s_{\mathrm{FD}}|$ on logarithmic scale for the four SPDE classes. The error is spatially unstructured and at machine precision: $\mathcal{O}(10^{-10})$ for second-order operators~(a,\,b) and $\mathcal{O}(10^{-9})$ for fourth-order operators~(c,\,d).}
\label{fig:error-2d}
\end{figure}

\FloatBarrier
\section{Summary}
\label{sec:summary}
We studied score-based diffusion models in infinite-dimensional separable Hilbert spaces using Malliavin calculus. By formulating the forward diffusion process as a linear SPDE driven by space--time coloured noise with a trace-class covariance operator, we ensured mathematical well-posedness across arbitrary spatial dimensions. Our derivation of the logarithmic derivative of the transition measure, the natural infinite-dimensional analogue of the score function, uses Malliavin calculus and an infinite-dimensional generalisation of the Bismut--Elworthy--Li formula, yielding a closed-form expression along Cameron--Martin directions without relying on finite-dimensional projections or approximations. This operator-theoretic approach preserves the intrinsic structure of Hilbert spaces and accommodates general trace-class operators, incorporating spatially correlated noise without assuming semigroup invertibility.

A key insight of our analysis is that the score (logarithmic derivative) is naturally defined only along directions in the Cameron--Martin space $\CMspace = \Ran(\gamma_t^{1/2})$, which is strictly smaller than $H$. While this may appear restrictive, it is in fact the maximal domain on which the score is meaningful; translations outside $\CMspace$ move the Gaussian measure to a mutually singular measure. In practical discretisations, this restriction becomes invisible as the projected Cameron--Martin space fills the finite-dimensional approximation space.

We validated the score formula numerically for four classes of linear SPDEs in one spatial dimension (spectral Galerkin discretisation with Dirichlet boundary conditions) and four classes in two spatial dimensions (Fourier spectral discretisation with periodic boundary conditions), the latter including both second- and fourth-order operators. In all cases, the Malliavin score agrees with finite-difference approximations to machine precision.



\acknowledgements

Daniele Venturi was supported by the U.S. Department of Energy (DOE) under grant DE\textendash SC0024563.

\appendix
\section{Malliavin Calculus}
\label{app:background}
In this appendix, we provide a brief overview of Malliavin calculus for linear SPDEs of the form \eqref{eq:spde}. To this end, let $W_t$ be a cylindrical Wiener process on a separable Hilbert space $U$. For any $\Phi \in L^2([0,t]; \mathcal{L}_{\mathrm{HS}}(U,H))$ we define the 
stochastic integral $\int_0^t \Phi(s)\,dW_s$ as
\begin{equation}
\int_0^t \Phi(s)\,dW_s := \sum_{j=1}^{\infty} \int_0^t \Phi(s)f_j\,dB_j(s),
\label{i1}
\end{equation}
where $\{f_j\}$ is an orthonormal basis of $U$ and $B_j(s) = \left\langle W_s, f_j \right\rangle_U$ are independent Brownian motions.
The integral \eqref{i1} is well-defined in $L^2(\Omega; H)$ and it satisfies It\^o's isometry
\begin{equation}
    \mathbb{E}\left[\left\|\int_0^t \Phi(s)\,dW_s\right\|_H^2\right] = \int_0^t \left\|\Phi(s)\right\|_{\mathrm{HS}}^2\,ds = \int_0^t \mathrm{tr}_U\left(\Phi(s)^*\Phi(s)\right)\,ds.
\end{equation}
For the SPDE \eqref{eq:spde}, we have $\Phi(s) = S(t-s)Q^{1/2}$, and the mild solution
\[
u(t) = S(t)u_0 + \int_0^t S(t-s)Q^{1/2}\,dW_s.
\]

\begin{lemma}
\label{lem:well-defined}
Under condition \eqref{eq:HS-condition}, the stochastic convolution $W_A(t) = \int_0^t S(t-s)Q^{1/2}\,dW_s$ is well-defined in $L^2(\Omega; H)$ with
\[
\mathbb{E}\left[\|W_A(t)\|_H^2\right] = \int_0^t \left\|S(t-s)Q^{1/2}\right\|_{\mathrm{HS}}^2\,ds = \mathrm{tr}_H(\gamma_t),
\]
where $\gamma_t$ is the Malliavin covariance operator \eqref{MCfin}.
\end{lemma}

\begin{proof}
By the It\^o isometry and the cyclic property of trace
\begin{align}
\mathbb{E}\left[\|W_A(t)\|_H^2\right] &= \int_0^t \left\|S(t-s)Q^{1/2}\right\|_{\mathrm{HS}}^2\,ds\nonumber\\
&= \int_0^t \mathrm{tr}_U\left(\left(Q^{1/2}\right)^*S(t-s)^*S(t-s)Q^{1/2}\right)\,ds\nonumber\\
&= \int_0^t \mathrm{tr}_H\left(S(t-s)Q^{1/2}\left(Q^{1/2}\right)^*S(t-s)^*\right)\,ds\nonumber\\
&= \mathrm{tr}_H\left(\int_0^t S(s)Q^{1/2}\left(Q^{1/2}\right)^*S(s)^*\,ds\right) = \mathrm{tr}_H\left(\gamma_t\right).\nonumber
\end{align}
This completes the proof.

\end{proof}

Let us now define cylindrical Wiener processes, the Cameron--Martin space, and cylindrical functionals of the Wiener process. For each $u\in U$ the cylindrical Wiener process on $U$ is the real-valued Brownian motion 
\begin{equation}
    W_t(u) = \left\langle W_t,u\right\rangle
\end{equation}
As is well known, the variance and covariance of $W_t(u)$ are, respectively
\begin{align}
\mathbb{E}\{W_t(u)^2\}=t\left\|u\right\|_U^2,\qquad 
\mathbb{E}\{W_t(u)W_t(v)\}=t\left\langle u,v\right\rangle_U.
\end{align}

\begin{definition}[Cameron--Martin space] The Cameron--Martin space is defined as
\[
H_W = L^2([0,t]; U)
\]
with inner product $\left\langle h, g \right\rangle_{H_W} = \int_0^t \left\langle h(s), g(s) \right\rangle_U\,ds.$
\end{definition}

\noindent
For $h \in H_W$, the Wiener integral
\[
W(h) = \int_0^t \left\langle h(s), dW_s \right\rangle_U
\]
is a Gaussian random variable with mean $\mathbb{E}[W(h)] = 0$ and variance $\mathbb{E}[W(h)^2] = \left\|h\right\|_{H_W}^2$.

\subsection{Cylindrical functionals and Malliavin Derivatives}
\label{app:a3}

Let $\mathcal{S}$ denote smooth cylindrical functionals of the form
\[
F = f(W(h_1), \ldots, W(h_n))
\]
where $h_i \in H_W$ and $f \in C_{\mathrm{pol}}^{\infty}(\mathbb{R}^n)$ (smooth functions with polynomial growth derivatives).

\begin{definition}[Malliavin derivative]
For $F \in \mathcal{S}$, the Malliavin derivative is the $H_W$-valued random variable
\[
\Mall_r F = \sum_{i=1}^n \frac{\partial f}{\partial x_i}(W(h_1), \ldots, W(h_n)) h_i(r), \quad r \in [0,t].
\]
\end{definition}

\noindent
Beyond the score computation pursued in this paper, Malliavin derivatives 
have been used to approximate polynomial nonlinearities in nonlinear 
SPDEs via Wick--Malliavin 
expansions~\citep{venturi2013wick}.

\begin{lemma}[Malliavin derivative of the solution to the SPDE \eqref{eq:spde}]
\label{lem:Dru-full}
Let $\mathbb{D}^{1,2}$ be the Sobolev space defined as the completion of $\mathcal{S}$ under the norm
\[
\left\|F\right\|_{1,2} = \left(\mathbb{E}\left[F^2\right] + \mathbb{E}\left[\int_0^t \left\|\Mall_r F\right\|_U^2\,dr\right]\right)^{1/2}.
\]
The mild solution of the linear SPDE \eqref{eq:spde}, i.e., \eqref{sol}, belongs to $\mathbb{D}^{1,2}(H)$ and its Malliavin derivative is given by
\[
\Mall_r u(t) = S(t-r)Q^{1/2}\mathbf{1}_{[0,t]}(r).
\]
\end{lemma}

\begin{proof}
Consider the perturbation $W^{\varepsilon}(s) = W(s) + \varepsilon\int_0^s \eta(\tau)\,d\tau$ for $\eta \in H_W$. The perturbed solution is
\[
u^{\varepsilon}(t) = S(t)u_0 + \int_0^t S(t-s)Q^{1/2}\,dW_s^{\varepsilon} = u(t) + \varepsilon\int_0^t S(t-s)Q^{1/2}\eta(s)\,ds.
\]
Taking the Fr\'echet derivative
\[
\left\langle \Mall u(t), \eta \right\rangle_{H_W} = \lim_{\varepsilon \to 0} \frac{u^{\varepsilon}(t) - u(t)}{\varepsilon} = \int_0^t S(t-s)Q^{1/2}\eta(s)\,ds.
\]
By the Riesz representation theorem, $\Mall_r u(t) = S(t-r)Q^{1/2}\mathbf{1}_{[0,t]}(r)$.
Finally, we verify $u(t) \in \mathbb{D}^{1,2}(H)$. To this end, we notice that
\[
\int_0^t \left\|S(t-r)Q^{1/2}\right\|_{\mathrm{HS}}^2\,dr < \infty,
\]
by condition \eqref{eq:HS-condition}.

\end{proof}

Hereafter, we characterise the Malliavin derivative of a function of the SPDE solution $u(t)$. 

\begin{theorem}[Chain rule]
\label{thm:chain-rule-full}
Let $\phi \in C_b^1(H)$ with bounded Fr\'echet derivative. Then for $F = \phi(u(t))$
\[
\Mall_r F = \left(Q^{1/2}\right)^*S(t-r)^*\nabla\phi(u(t))\mathbf{1}_{[0,t]}(r).
\]
\end{theorem}

\begin{proof}
By the chain rule for Fr\'echet derivatives
\[
\Mall_r(\phi(u(t))) = \nabla\phi(u(t))[\Mall_r u(t)] = \left\langle \nabla\phi(u(t)), \Mall_r u(t) \right\rangle_H.
\]
Since $\Mall_r u(t) = S(t-r)Q^{1/2}\mathbf{1}_{[0,t]}(r)$ takes values in $\mathcal{L}_{\mathrm{HS}}(U,H)$, we need the adjoint action. For $v \in U$ we have
\[
\left\langle \Mall_r F, v \right\rangle_U = \left\langle \nabla\phi(u(t)), S(t-r)Q^{1/2}v \right\rangle_H = \left\langle \left(Q^{1/2}\right)^*S(t-r)^*\nabla\phi(u(t)), v \right\rangle_U.
\]
Therefore $\Mall_r F = \left(Q^{1/2}\right)^*S(t-r)^*\nabla\phi(u(t))\mathbf{1}_{[0,t]}(r)$.

\end{proof}

\subsection{Skorokhod Integral}
\label{app:a5}
For $\eta \in H_W$, let us define the operator 
$\Mall u(t): H_W \to H$ as
\[
\Mall u(t)\eta = \int_0^t S(t-s)Q^{1/2}\eta(s)\,ds
\]
with adjoint $\Mall u(t)^*: H \to H_W$
\[
\Mall u(t)^*h = \left(Q^{1/2}\right)^*S(t-\cdot)^*h\,\mathbf{1}_{[0,t]}(\cdot).
\]

\begin{definition}[Skorokhod integral]
The Skorokhod integral $\delta: \mathrm{Dom}(\delta) \subset L^2(\Omega; H_W) \to L^2(\Omega)$ is the adjoint of $\Mall$
\begin{equation}
\mathbb{E}[\left\langle \Mall F, v \right\rangle_{H_W}] = \mathbb{E}[F\delta(v)]
\label{SK}
\end{equation}
for all $F \in \mathbb{D}^{1,2}$ and $v \in \mathrm{Dom}(\delta)$.
\end{definition}

It can be shown that for deterministic $v \in H_W$, the Skorokhod integral 
reduces to the It\^o integral
\[
\delta(v) = \int_0^t \left\langle v(s), dW_s \right\rangle_U = W(v).
\]

\begin{proposition}[Integration by parts]
Let $\gamma_t$ be the Malliavin covariance operator \eqref{MCO}. For $F = \left\langle u(t), h \right\rangle_H$ with $h \in \mathrm{Ran}(\gamma_t)$, and \(v(r) = (Q^{1/2})^* S(t-r)^* \gamma_t^{-1} h\,\mathbf{1}_{[0,t]}(r)\) we have
\[
\mathbb{E}[F\delta(v)] = \left\langle h, h \right\rangle_H.
\]
\end{proposition}


We have seen that the solution to the SPDE \eqref{eq:spde} is Gaussian with mean $S(t)u_0$ and covariance operator \eqref{MCO}. This covariance operator satisfies the following recursion.

\begin{lemma}[Covariance recursion]
For all $s, r \geq 0$
\[
\gamma_{t+s} =\gamma_t  + S(t)\gamma_s S(t)^*.
\]
Consequently, $S(t)\mathrm{Ran}(\gamma_s) \subset \mathrm{Ran}(\gamma_{t+s} )$ for all $t, s \geq 0$.
\end{lemma}

\begin{proof}
By direct computation,
\begin{align}
\gamma_{t+s} &= \int_0^{t+s} S(r)Q^{1/2}\left(Q^{1/2}\right)^*S(r)^*\,dr\nonumber\\
&= \int_0^t S(r)Q^{1/2}\left(Q^{1/2}\right)^*S(r)^*\,dr + \int_t^{t+s} S(r)Q^{1/2}\left(Q^{1/2}\right)^*S(r)^*\,dr\nonumber\\
&= \gamma_t + S(t)\left(\int_0^s S(\sigma)Q^{1/2}\left(Q^{1/2}\right)^*S(\sigma)^*\,d\sigma\right)S(t)^*\nonumber\\
&= \gamma_t  + S(t)\gamma_s  S(t)^*\nonumber.
\end{align}
The range inclusion follows immediately.

\end{proof}


\end{document}